\documentclass[final,3p,times]{elsarticle}
\usepackage{lineno,hyperref}
\usepackage{amssymb,amsmath,amsthm}
\usepackage{hyperref}
\usepackage{enumitem}

\usepackage[OT1]{fontenc} 

\newtheorem{theorem}{Theorem}[section] 
\newtheorem{lemma}[theorem]{Lemma}
\newtheorem{definition}{Definition}[section]
\newtheorem{proposition}[theorem]{Proposition}
\newtheorem{remark}[theorem]{Remark}

\journal{}

\begin{document}

\begin{frontmatter}

\title{Nonlocal Lazer-McKenna type problem perturbed by the Hardy's potential and its parabolic equivalence}

\author[mymainaddress]{Masoud Bayrami-Aminlouee\corref{mycorrespondingauthor}}
\cortext[mycorrespondingauthor]{Corresponding author}
\ead{masoud.bayrami1990@student.sharif.edu}

\author[mymainaddress]{Mahmoud Hesaaraki}
\ead{hesaraki@sharif.edu}
\address[mymainaddress]{Department of Mathematical Sciences, Sharif University of Technology, Tehran, Iran}

\author[mymainaddress1,mymainaddress3,mymainaddress4]{Mohamed Karim Hamdani}
\address[mymainaddress1]{Military School of Aeronautical Specialities, Sfax, Tunisia}
\address[mymainaddress3]{Mathematics Department, University of Sfax, Faculty of Science of Sfax, Sfax, Tunisia}
\address[mymainaddress4]{Science and technology for defense Laboratory LR19DN01, Military Research Center, Aouina, Tunisia}
\ead{hamdanikarim42@gmail.com}

\author[mymainaddress2]{Nguyen Thanh Chung}
\address[mymainaddress2]{Department of Mathematics, Quang Binh University, 312 Ly Thuong Kiet, Dong Hoi, Quang Binh, Vietnam}
\ead{ntchung82@yahoo.com}

\begin{abstract}
In this paper, we study the effect of Hardy potential on the existence or non-existence of solutions to the following fractional problem involving a singular nonlinearity:
$$
\begin{cases}
(-\Delta )^s u = \lambda \dfrac{u}{|x|^{2s}} +\dfrac{\mu}{u^{\gamma}}+f & \quad \mathrm{in} \,\, \Omega,\\ u>0 & \quad \mathrm{in} \,\, \Omega, \\ u=0 & \quad \mathrm{in} \,\, \big(\mathbb{R}^N \setminus \Omega \big).
\end{cases}
$$
Here $0 <s<1$, $\lambda>0$, $\gamma>0$, and $\Omega \subset \mathbb{R}^N$ ($N > 2s$) is a bounded smooth domain such that $0 \in \Omega$. Moreover, $0 \leq \mu,f \in L^1(\Omega)$. For $0< \lambda \leq \Lambda_{N,s}$, $\Lambda_{N,s}$ being the best constant in the fractional Hardy inequality, we find the necessary and sufficient condition for the existence of a positive weak solution to the above problem with respect to the data $\mu$ and $f$. Also, for a regular datum of $f$ and with suitable assumptions, we have some existence and uniqueness results and calculate the rate of the growth of solutions. Moreover, we mention a non-existence and a complete blow-up result for the case $\lambda > \Lambda_{N,s}$. Besides, we consider the parabolic equivalence of the above problem in the case $\mu \equiv 1$,  and some suitable $f(x,t)$, i.e.
$$
\begin{cases}
u_t+(-\Delta )^s u = \lambda \dfrac{u}{|x|^{2s}} +\dfrac{1}{u^{\gamma}}+f(x,t) & \quad \mathrm{in} \,\, \Omega \times (0,T), \\ u>0 & \quad \mathrm{in} \,\, \Omega \times (0,T), \\ u =0 & \quad \mathrm{in} \,\, (\mathbb{R}^N \setminus \Omega) \times (0,T), \\ u(x,0)=u_0 & \quad \mathrm{in} \,\, \mathbb{R}^N,
\end{cases}
$$
where $u_0 \in X_0^{s}(\Omega)$ satisfies an appropriate cone condition. In the case $0<\gamma \leq 1$, or $\gamma>1$, with $2s(\gamma-1)<(\gamma+1)$, we show the existence of a unique solution, for any $0< \lambda < \Lambda_{N,s}$, and prove a stabilization result for certain range of $\lambda$.
\end{abstract}


\end{frontmatter}

\section{Introduction}
We study on the existence and non-existence of positive solutions to the following singular elliptic problem: 
\begin{equation}
\label{Eq1}
\begin{cases}
(-\Delta )^s u = \lambda \dfrac{u}{|x|^{2s}} + \dfrac{\mu}{u^{\gamma}} + f & \quad \mathrm{in} \,\, \Omega,\\ u>0 & \quad \mathrm{in} \,\, \Omega, \\ u=0 & \quad \mathrm{in} \,\, \big(\mathbb{R}^N \setminus \Omega \big).
\end{cases}
\end{equation}
Here $\Omega \subset \mathbb{R}^N$, $N > 2s$, is an open bounded domain with smooth boundary such that $0 \in \Omega$, $ s \in (0,1)$, $ \lambda >0$, and $\gamma>0$. Moreover, $0 \leq \mu, f \in L^1(\Omega)$.

We will prove that for $0 < \lambda \leq \Lambda_{N,s}$, $\Lambda_{N,s}=\frac {4^s\Gamma^2(\frac{N+2s}4)}{\Gamma^2(\frac{N-2s}4)}$ being the best constant in the fractional Hardy inequality, the above problem has a solution if and only if $ \mu \in L^1(\Omega, \delta^{s(1-\gamma)} \, dx)$, $\delta(x)=\mathrm{dist}(x,\partial \Omega)$, and the datum of $f$ satisfies the following integrability condition:
\begin{equation*} 
\int_{\Omega} f(x) |x|^{-\beta} \, dx < + \infty, 
\end{equation*}
where the constant $\beta=\beta(N,s,\lambda)$ will be defined later in Lemma \ref{Lem1}. In this lemma, we will see that any supersolution to \eqref{Eq1} is unbounded near the origin and the nature of this unboundedness is like $u(x) \gtrsim |x|^{-\beta}$ in some open ball centered at the origin.

Also, we will see that there is no positive very weak (distributional) solution for the case $\lambda > \Lambda_{N,s}$. This notion of the solution, which we consider for the non-existence result, is local in nature and we just ask the regularity needed to give distributional sense to the equation (similar to what is done in articles \cite{MR2545987, MR3258136}). Moreover, this non-existence result is strong in the sense that a complete blow-up phenomenon occurs. By complete blow-up phenomenon, we mean that the solutions to the approximating problems (with the bounded weights $(|x|^{2s}+\epsilon)^{-1}$ and $(u+\epsilon)^{-\gamma}$ instead of the terms $|x|^{-2s}$ and $u^{-\gamma}$, respectively) tend to infinity for every $x \in \Omega$, as $0<\epsilon \downarrow 0$.

In the above problem, $(-\Delta)^{s} $ stands for the fractional Laplacian operator, i.e.
\begin{equation*} 
\begin{aligned}
(-\Delta )^{s} u(x) &= C_{N,s} \, \mathrm{P.V.} \int_{\mathbb{R}^N} \frac{u(x)-u(y)}{|x-y|^{N+2s}} \, dy \\
&= C_{N,s} \lim_{\epsilon \to 0^+} \int_{|x-y| \geq \epsilon} \frac{u(x)-u(y)}{|x-y|^{N+2s}} \, dy, \qquad u \in \mathcal{S}(\mathbb{R}^N),
\end{aligned}
\end{equation*}
where $\mathrm{P.V.}$ is a commonly used abbreviation for the Cauchy principal value and is defined by the latter equation. Also, $\mathcal{S}(\mathbb{R}^N)$ denotes the Schwartz space (space of ``rapidly decreasing functions'' on $\mathbb{R}^N$) and $C_{N,s}= \frac{4^s \Gamma(\frac{N}{2}+s)}{\pi^{\frac{N}{2}} |\Gamma(-s)|}$, is the normalization constant such that 
\begin{equation*}  
(-\Delta )^s u = \mathcal{F}^{-1} \big(|\xi|^{2s} \hat{u}(\xi) \big).
\end{equation*}
Here $\Gamma$ denotes the Gamma function, and $\mathcal{F}u =\hat{u}$ is the Fourier transform of $u$. By restricting the fractional Laplacian operator to act only on smooth functions that are zero outside $\Omega$, we have the restricted fractional Laplacian $(-\Delta_{|_{\Omega}})^s$. For this operator, the best alternative to the Dirichlet boundary condition is $ u \equiv 0 $ in $\big(\mathbb{R}^N \setminus \Omega \big)$. For more details about fractional Laplacian, see \cite{MR3967804, MR2944369, MR2270163}.

Over the past decades, there has been much focus and also a vast literature about singular problems. Singularities appear in almost all fields of mathematics like differential geometry and partial differential equations. Singularities are the qualitative side of mathematics, and understanding of singularities always leads to a more detailed picture of the objects mathematics is dealing with, \cite{MR3468562}. Many more details and references for the singular elliptic problems can be found in \cite{MR2488149}. 

One famous type of singularities are the singularity of Hardy type, which is related to the inequality of the same name, and there are various generalizations of it. The well-known classical Hardy inequality is as follows:
\begin{equation*} 
\int_{\Omega} |\nabla u|^p \, dx \geq \Big(\frac{N-p}{p} \Big)^p \int_{\Omega} \frac{|u|^p}{|x|^p} \, dx, \qquad u \in W_0^{1,p}(\Omega),
\end{equation*}
where $\Omega \subset \mathbb{R}^N$, containing the origin, is a bounded domain and $1 \leq p <N$, \cite{MR2048513, MR1769903}. The constant $\big(\frac{N-p}{p}\big)^p$ is optimal and it is not attained in $W_0^{1,p}(\Omega)$, meaning that the continuous embedding $W_0^{1,p}(\Omega) \hookrightarrow L^p (\Omega,|x|^{-p} \, dx )$ is not compact. The intention of analyzing Hardy singularities has come from its widespread use in different branches of science. For details and references about the enormous literature for this topic, see the more recent book \cite{HardyLerayBook} and chapter 1 of \cite{Biccari}. Due to these motivations, over the past few decades, the study of general singularities has been considered.

In the pioneering works, \cite{MR742415, MR799330}, Baras and Goldstein studied the following singular Cauchy-Dirichlet heat problem in $\Omega=\mathbb{R}^N$ or else $\Omega$ to be a bounded smooth domain containing $B_1(0)= \{ x \in \mathbb{R}^N: \, \|x\| <1 \}$.
\begin{equation}
\label{EqBG}
\begin{cases}
\dfrac{\partial u}{\partial t} - \Delta u = V(x) u +f(x,t) \quad & (x,t) \in \Omega \times (0, \infty) \\
u(x,t)=0 \quad & (x,t) \in \partial \Omega \times (0, \infty) \\
u(x,0)=u_0(x) \quad & x \in \Omega.
\end{cases}
\end{equation}
Authors assume that $f$ and $u_0$ are non-negative and $0 \leq V \in L^{\infty}(\Omega \setminus B_{\epsilon}(0))$, for each $\epsilon>0$, but $V$ is singular at the origin. They say that $V$ is too singular if $V(x) > \frac{C^*(N)}{|x|^2}$ near $x = 0$, while $V$ is not too singular if $V(x) \leq \frac{C^*(N)}{|x|^2}$ near $x=0$. Here $C^*(N)=\frac{(N-2)^2}{4}$ is the sharp constant in the following Hardy inequality:
\begin{equation*}  
C^*(N) \int_{\Omega} \frac{u^2}{|x|^2} \, dx \leq  \int_{\Omega} | \nabla u |^2 \, dx, \qquad \forall u \in H_0^1(\Omega). 
\end{equation*}

In the not too singular potential case, they found the necessary and sufficient condition for the existence of a non-negative distributional solution to problem \eqref{EqBG}. Moreover, they obtained this solution as the limit of the solutions to the following approximate problem.
\begin{equation*} 
\begin{cases}
\dfrac{\partial u_n}{\partial t} - \Delta u_n = V_n(x) u_n +f(x,t) \quad & (x,t) \in \Omega \times (0, \infty) \\
u_n(x,t)=0 \quad & (x,t) \in \partial \Omega \times (0, \infty) \\
u_n(x,0)=u_0(x) \quad & x \in \Omega,
\end{cases}
\end{equation*}
where $V_n(x)=\min \{V(x),n\}$. Also, for the too singular potential case, they showed that the problem has no solution even in the sense of distributions, and an instantaneous complete blow-up phenomenon occurs. Namely, $u_n(x,t) \to + \infty$ for all $(x,t) \in \Omega \times (0,T)$ as $n \to \infty$.

In problem \eqref{Eq1}, the singular term $\frac{\lambda}{|x|^{2s}}$ is related to the following fractional Hardy inequality:
\begin{equation}
\label{Eq3.05}
\Lambda_{N,s} \int_{\mathbb{R}^{N}}\frac{|u(x)|^2}{|x|^{2s}}\,dx
\leq \int_{\mathbb{R}^{N}}|(-\Delta )^{\frac{s}{2}} u(x)|^2\,dx \qquad \forall u\in C_{c}^{\infty}(\mathbb{R}^N),
\end{equation}
where $N>2s$, $ s \in (0,1)$ and the constant $\Lambda_{N,s}=\frac {4^s\Gamma^2(\frac{N+2s}4)}{\Gamma^2(\frac{N-2s}4)}$ is optimal, \cite{MR1717839}. Problem \eqref{Eq1} is motivated by the papers \cite{MR2592976, MR3450747} in which the authors proved the existence of solutions to the following Lazer-McKenna type problem:
\begin{equation*} 
\begin{cases}
-\Delta u = \dfrac{\mu}{u^{\gamma}} \quad & \mathrm{in} \,\, \Omega \\
u>0 \quad & \mathrm{in} \,\, \Omega \\
u=0 \quad & \mathrm{on} \,\, \partial \Omega,
\end{cases}
\end{equation*}
where $\Omega$ is a bounded domain of $\mathbb{R}^N$, $N\geq2$, $\gamma>0$ and $\mu$ a general Radon measure in $\Omega$. Also, see the papers \cite{MR3356049, MR3639996, MR3797738, MR3943307, MR1037213, MR3323892, MR3489386, MR3712944, MR4244929, MADD1} for more related problems. These types of problems have been extensively studied for their relations with some physical phenomena in the theory of pseudoplastic fluids, \cite{MR564014}.

In \cite{MR3356049} Barrios, Bonis, Medina and Peral studied the solvability of the following superlinear problem: 
\begin{equation*} 
\begin{cases}
(-\Delta )^s u = \lambda \dfrac{f(x)}{u^{\gamma}}+Mu^p & \quad \mathrm{in} \,\, \Omega, \\
u>0 & \quad \mathrm{in} \,\, \Omega, \\ 
u=0 & \quad \mathrm{in} \,\, \big(\mathbb{R}^N \setminus \Omega \big). 
\end{cases}
\end{equation*}
More precisely, for the case $M=0$ and $f \geq 0$, they proved the existence of a positive solution for every $\gamma>0$ and $\lambda>0$. Moreover, in the case $M=1$ and $f \equiv 1$, they found a threshold $\Lambda$ such that there exists a solution for every $0 < \lambda < \Lambda$, and there does not for $ \lambda > \Lambda$. Also in \cite{MR3466525} authors considered the similar superlinear problem with the critical growth, namely when $p=2_s^*-1=\frac{N+2s}{N-2s}$, and with a singular nonlinearity in the form $u^{-q}, q \in (0,1)$.

In the detailed article \cite{MR3479207}, Abdellaoui, Medina, Peral, and Primo studied the effect of the Hardy potential on the existence and summability of the solutions to a class of fractional Laplacian problems. We will use the essential tool introduced in this article, i.e., the weak Harnack's inequality, which they proved it by following the classical Moser and Krylov-Safonov idea. Also, we will take advantage of some of Calder\'on-Zygmund properties of solutions. See \cite[Section 4]{MR3479207} for the effect of the Hardy potential in some Calder\'on-Zygmund properties for the fractional Laplacian.

For the similar parabolic equivalence of \eqref{Eq1}, in \cite{MR3842325}, Giacomoni, Mukherjee and Sreenadh investigated the existence and stabilization results for the following parabolic equation involving the fractional Laplacian with singular nonlinearity:
\begin{equation*} 
\begin{cases}
u_t + (-\Delta )^s u = u^{-q} + f(x,u) &   \mathrm{in} \,\, \Omega \times (0,T), \\
u(x,0)=u_0(x) &   \mathrm{in} \,\, \mathbb{R}^N, \\
u(x,t)>0 &   \mathrm{in} \,\, \Omega \times (0,T), \\ 
u(x,t)=0 &   \mathrm{in} \,\, \big(\mathbb{R}^N \setminus \Omega \big) \times (0,T). 
\end{cases}
\end{equation*}
Under suitable assumptions on the parameters and datum, they studied the related stationary problem and then using the semi-discretization in time with the implicit Euler method, they proved the existence and uniqueness of the weak solution. It is worth noting that in \cite{MR2891356, MR3341461}, the authors have shown the same results for the local version of this problem for the general $p$-Laplacian case. Also for some of the recent papers on the optimal regularity results see \cite{Giacomoni001, Giacomoni002}.

The rest of the paper is as follows. In section \ref{Section2}, after introducing the functional setting we will outline our existence and non-existence theorems. Especially, we will have a theorem about the necessary and sufficient condition for the existence of a solution to problem \eqref{Eq1} in the case $\lambda \leq \Lambda_{N,s}$, and a non-existence theorem in the case $\lambda > \Lambda_{N,s}$. In section \ref{Section3}, we will provide proof of our existence theorems. In section \ref{Section3.5}, we will have some uniqueness results. Also, concerning uniqueness, with some regular assumptions on $\mu$ and $f$, we will show the existence and uniqueness of another notion of a solution so-called entropy solution for the case $0<\gamma\leq 1$. Besides, we will mention a theorem about the rate of the growth of solutions to problem \eqref{Eq1}. Finally, in section \ref{Section4}, we will consider the parabolic version of problem \eqref{Eq1} in the special case $\mu \equiv 1$. Firstly, with the assumptions $0<\gamma \leq 1$, or $\gamma>1$, and $2s(\gamma-1)<(\gamma+1)$, we will show the existence of a unique solution for $0< \lambda < \Lambda_{N,s}$ and secondly, we will prove the stability for some range of $\lambda$. That is, we will find a positive constant $\lambda_*=\lambda_*(N,s) < \Lambda_{N,s}$ such that for any $\lambda \in (0, \lambda_*)$, the solution to the parabolic problem converges to the unique solution of its stationary problem, as $t \to \infty$.

\section{Functional setting and existence, non-existence and blow-up results}
\label{Section2}
Let $ 0 < s <1 $, $ 1 \leq p < \infty $, and $\Omega$ be a bounded domain in $\mathbb{R}^N $. Also, let $D_{\Omega}= \mathbb{R}^N \times \mathbb{R}^N \setminus \Omega^{c} \times \Omega^{c} $, with $\Omega^{c}=\mathbb{R}^N \setminus \Omega$. We define the following Banach space
\begin{equation*} 
X^{s,p}(\Omega) =\Big\{ u: \mathbb{R}^N \to \mathbb{R} \, \mathrm{measurable}, u|_{\Omega} \in L^p(\Omega), \iint_{D_{\Omega}} \frac{|u(x)-u(y)|^p}{|x-y|^{N+ps}} dxdy < \infty \Big\},
\end{equation*}
endowed with the norm:
\begin{equation}
\label{Eq3.1}
\|u\|_{X^{s,p}(\Omega)} = \Bigg( \int_{\Omega} |u|^p \, dx + \iint_{D_{\Omega}} \frac{|u(x)-u(y)|^p}{|x-y|^{N+ps}} \, dxdy \Bigg)^{\frac{1}{p}}. 
\end{equation}
In the case $p=2$, we denote by $X^s(\Omega)$ the space $X^{s,2}(\Omega)$ which is a Hilbert space with the following inner product:
\begin{equation*} 
\langle u,v \rangle_{X^s(\Omega)} = \int_{\Omega} uv \, dx + \iint_{D_{\Omega}} \frac{(u(x)-u(y))(v(x)-v(y))}{|x-y|^{N+2s}} \, dxdy. 
\end{equation*}
Moreover, we define $ X_0^{s,p}(\Omega) $ as the closure of $C_0^{\infty}(\Omega)$ in $X^{s,p}(\mathbb{R}^N)$. Equivalently, it can be shown that 
\begin{equation*} 
X_0^{s,p}(\Omega) = \Big\{ u \in X^{s,p}(\mathbb{R}^N) \, : \, u=0 \,\, \mathrm{a.e. \,\, in} \,\, (\mathbb{R}^N \setminus \Omega ) \Big\}.
\end{equation*}
It is easy to see that: 
\begin{equation*} 
\Bigg( \int_{\mathbb{R}^N} \int_{\mathbb{R}^N} \frac{|u(x)-u(y)|^p}{|x-y|^{N+ps}} \, dxdy \Bigg)^{\frac{1}{p}} = \Bigg( \iint_{D_{\Omega}} \frac{|u(x)-u(y)|^p}{|x-y|^{N+ps}} \, dxdy \Bigg)^{\frac{1}{p}}, \,\,\, \forall u \in X_0^s(\Omega). 
\end{equation*}
This equality defines an equivalent norm for $X_0^{s,p}(\Omega)$ with \eqref{Eq3.1}. We denote it by
\begin{equation*} 
\|u\|_{X_0^{s,p}(\Omega)} = \Bigg( \iint_{D_{\Omega}} \frac{|u(x)-u(y)|^p}{|x-y|^{N+ps}} \, dxdy \Bigg)^{\frac{1}{p}}.
\end{equation*}
It is worth noticing that, the continuous embedding of $X_0^{s_2}(\Omega)$ into $X_0^{s_1}(\Omega)$, holds for any $s_1<s_2$ (see, e.g. \cite[Proposition 2.1]{MR2944369}). Besides, for the Hilbert space case, we have
\begin{equation}
\label{EqEqEqOO.}
\|u\|_{X_0^s(\Omega)}^2 = 2 C_{N,s}^{-1} \| (-\Delta)^{\frac{s}{2}} u\|_{L^2(\mathbb{R}^N)}^2,
\end{equation}
where $C_{N,s}$ is the normalization constant in the definition of $(-\Delta)^s$. Thus Hardy inequality \eqref{Eq3.05} also can be written as follows:
\begin{equation*} 
\Lambda_{N,s} \int_{\mathbb{R}^{N}}\frac{|u(x)|^2}{|x|^{2s}}\,dx
\leq \frac{C_{N,s}}{2} \iint_{D_{\Omega}} \frac{|u(x)-u(y)|^2}{|x-y|^{N+2s}}\,dxdy, \qquad \forall u\in X_0^s(\Omega).
\end{equation*}
For the proofs of the above facts see \cite[Subsection 2.2]{MR2879266} and \cite{MR2944369}. Also, see \cite[Section 2]{MR3341104}.

The following continuous embedding will be used in this paper.
\begin{equation}
\label{Eq3.350}
X_0^{s,p}(\Omega) \hookrightarrow L^q(\Omega), \qquad \forall q \in [1, p_s^{*}],
\end{equation}
where $p_s^{*}=\frac{pN}{N-ps}$ is the Sobolev critical exponent. Moreover, this embedding is compact for $1 \leq q < p_s^{*}$. See \cite[Theorem 6.5 and Theorem 7.1]{MR2944369}.

Also we denote by $X_{\mathrm{loc}}^{s,p}(\Omega)$, the set of all functions $u$ such that $u \phi \in X_0^{s,p}(\Omega)$ for any $\phi \in C_c^{\infty}(\Omega)$. When we say $\{u_n\} \subset X_{\mathrm{loc}}^{s,p}(\Omega)$ is bounded, we mean that $\{\phi u_n\} \subset X_0^{s,p}(\Omega)$ is bounded for any fixed $\phi \in C_c^{\infty}(\Omega)$. 
 
Since we are dealing with the non-local operator $(-\Delta)^s$, the following class of test functions will be used for defining the weak solution to problem \eqref{Eq1}.
\begin{equation*} 
\mathcal{T}(\Omega)= \Bigg \{ \phi:\mathbb{R}^N \to \mathbb{R} \,\, \Bigg| \,\, \begin{aligned} & (-\Delta)^s \phi = \varphi, \,\, \varphi \in L^{\infty}(\Omega) \cap C^{0,\alpha}(\Omega), \,\, 0 < \alpha <1, \\
& \,\, \phi = 0 \,\, \mathrm{in} \,\, (\mathbb{R}^N \setminus \Omega) \end{aligned} \Bigg\}. 
\end{equation*}
It can be shown that $\mathcal{T}(\Omega) \subset X_0^s(\Omega) \cap L^{\infty}(\Omega) \cap C^{0,s}(\Omega)$. See \cite{MR3479207}, where this class of test functions is used for dealing with problem \eqref{Eq1}. Moreover, every $\phi \in \mathcal{T}(\Omega)$ is a strong solution to the equation $(-\Delta)^s \phi = \varphi$, and for every $\phi \in \mathcal{T}(\Omega)$  there exists a constant $\beta \in (0,1)$ such that $\frac{\phi}{\delta^s} \in C^{0,\beta}(\Omega)$. See \cite{MR3168912}.

It is easy to check that for $u\in X_0^s(\Omega)$ and $ \phi \in \mathcal{T}(\Omega)$:
\begin{equation}
\label{Eq3.349n}
\begin{aligned}
2 C_{N,s}^{-1} \int_{\mathbb{R}^N} u (-\Delta)^{s} \phi \, dx &= 2 C_{N,s}^{-1} \int_{\mathbb{R}^N} (-\Delta)^{\frac{s}{2}} u (-\Delta)^{\frac{s}{2}} \phi \, dx \\
& = \iint_{D_{\Omega}} \frac{(u(x)-u(y))(\phi(x)-\phi(y))}{|x-y|^{N+2s}} \, dxdy.
\end{aligned}
\end{equation}
One can show that $ (-\Delta)^s : X_0^s(\Omega) \to X^{-s}(\Omega) $ is a continuous strictly monotone operator, where $X^{-s}(\Omega)$ indicates the dual space of $X_0^s(\Omega)$. 

\begin{definition}
We say that $u \in L^1(\Omega)$ is a very weak (distributional) supersolution (subsolution) to 
\begin{equation*} 
(-\Delta )^s u = g(x,u) \quad \mathrm{in} \,\, \Omega,
\end{equation*}
if $g(x,u) \in L^1(\Omega)$, $u \equiv 0$ in $\big(\mathbb{R}^N \setminus \Omega\big)$ and $(-\Delta )^s u \geq (\leq) g(x,u)$ in the weak sense, i.e.
\begin{equation*} 
\int_{\mathbb{R}^N} u (-\Delta )^s \phi \, dx \geq (\leq) \int_{\Omega} g(x,u) \phi \, dx,
\end{equation*}
for all non-negative $\phi \in \mathcal{T}(\Omega)$. If $u$ is a very weak (distributional) supersolution and subsolution, then we say that $u$ is a very weak (distributional) solution.
\end{definition}

\begin{definition}
We say that $u \in X_0^s(\Omega)$ is a weak energy supersolution (subsolution) to 
\begin{equation*} 
(-\Delta )^s u = g(x,u) \quad \mathrm{in} \,\, \Omega,
\end{equation*}
if $g(x,u) \in X_0^s(\Omega)$, $u \equiv 0$ in $\big(\mathbb{R}^N \setminus \Omega\big)$ and $(-\Delta )^s u \geq (\leq) g(x,u)$ in the weak sense, i.e.
\begin{equation*} 
\int_{\mathbb{R}^N} u (-\Delta )^s \phi \, dx \geq (\leq) \int_{\Omega} g(x,u) \phi \, dx,
\end{equation*}
for all non-negative $\phi \in X_0^s(\Omega)$. If $u$ is a weak energy supersolution and subsolution, then we say that $u$ is a weak energy solution.
\end{definition}

\begin{definition}
Assume $0 \leq \mu, f \in L^1(\Omega)$. We say that $u$ is a weak solution to problem \eqref{Eq1} if 
\begin{itemize}
\item
$u \in L^1(\Omega)$, and for every $K \Subset \Omega$, there exists $C_K>0$ such that $u(x) \geq C_K$ a.e. in $K$ and also $u \equiv 0$ in $\big(\mathbb{R}^N \setminus \Omega \big)$;
\item
Equation \eqref{Eq1} is satisfied in the weak sense, i.e.
\begin{equation}
\label{Eq3.4}
\int_{\mathbb{R}^N} u (-\Delta)^{s} \phi \, dx = \lambda \int_{\Omega} \frac{u \phi}{|x|^{2s}} \, dx + \int_{\Omega} \frac{\mu \phi}{u^{\gamma}} \, dx  + \int_{\Omega} f \phi \,dx, \quad \forall \phi \in \mathcal{T}(\Omega),
\end{equation}
and also together with these extra assumptions that the first and second terms on the right-hand side of the above equality be finite for any $\phi \in \mathcal{T}(\Omega)$. The well-posedness of the first and second terms on the right-hand side will be clear after the construction of solution.
\end{itemize}
\end{definition}  

\begin{remark}
\label{RemarkADD01}
Notice that plugging in the test function $\phi = \psi_{1,s}$ in \eqref{Eq3.4}, where $\psi_{1,s}$ is the normalized first eigenfunction associated with first eigenvalue $\lambda_{1,s}$ of $(-\Delta)^s$ in $X_0^s(\Omega)$, i.e.
\begin{equation*} 
\begin{cases}
(-\Delta)^s \psi_{1,s} = \lambda_{1,s} \psi_{1,s}  & \quad \text{in} \,\, \Omega, \\
\psi_{1,s}=0 & \quad \text{in} \,\, (\mathbb{R}^N \setminus \Omega), \\
0< \psi_{1,s} \in X_0^s(\Omega) \cap L^{\infty}(\Omega), \\
\|\psi_{1,s}\|_{L^2(\Omega)}=1, 
\end{cases}
\end{equation*}
(see for instance \cite[Proposition 9]{MR3002745}) and also noting that there exist $l_1,l_2>0$ such that $l_1 \delta^s(x) \leq \psi_{1,s}(x) \leq l_2 \delta^s(x)$, for a.e. $x \in \Omega$, \cite{MR3168912}, we obtain that the solution $u$ necessary satisfies: 
\begin{equation*} 
\int_{\Omega} \frac{\mu}{u^{\gamma}} \, \delta^s  dx < + \infty. 
\end{equation*}
Moreover, since by using Comparison Principle for the fractional Laplacian, and by Hopf's Lemma, $u \geq c \delta^s$ a.e. in $\Omega$, (see for example \cite{MR4065090} or \cite[Lemma 4.2]{MR3231059}) therefore 
\begin{equation*} 
\int_{\Omega} \frac{\mu}{\delta^{s(\gamma-1)}} \,dx < + \infty. 
\end{equation*}
\end{remark}

As an analysis of the linear case with Hardy potential, firstly, we gather the following lemmas.

\begin{lemma} 
\label{Lem1}
Let $\lambda \leq \Lambda_{N,s}$. Assume that $u$ is a non-negative function defined in $\Omega$ such that $ u \not \equiv 0$, $ u \in L^1(\Omega)$, $\frac{u}{|x|^{2s}} \in L^1(\Omega)$ and $u \geq 0$ in $ (\mathbb{R}^N \setminus \Omega)$. If $u$ satisfies $(-\Delta)^s u -\lambda \dfrac{u}{|x|^{2s}} \geq 0$ in the weak sense in $\Omega$, then there exists $\delta>0$, and a constant $C = C(N, \delta)$ such that
\begin{equation*} 
u \geq C |x|^{-\beta}, \qquad \mathrm{in} \,\, B_{\delta}(0),
\end{equation*}
where $\beta= \frac{N-2s}{2}-\alpha$ and $\alpha$ is given by the identity
\begin{equation}
\label{Eq3...01}
\lambda = \frac{4^{s} \Gamma(\frac{N+2s+2\alpha}{4}) \Gamma(\frac{N+2s-2\alpha}{4})}{\Gamma(\frac{N-2s+2\alpha}{4}) \Gamma(\frac{N-2s-2\alpha}{4})}.
\end{equation}
\end{lemma}

\begin{lemma}
\label{Lem2}
Let $\lambda \leq \Lambda_{N,s}$. Assume that $u$ is a positive very weak solution to
\begin{equation*} 
\begin{cases}
(-\Delta )^s u - \lambda \dfrac{u}{|x|^{2s}} = g & \quad \mathrm{in} \,\, \Omega, \\ u>0 & \quad \mathrm{in} \,\, \Omega, \\ u=0 & \quad \mathrm{in} \,\, \big(\mathbb{R}^N \setminus \Omega \big),
\end{cases}
\end{equation*}
with $g \in L^1(\Omega)$ and $g \geq 0$. Then necessarily $g |x|^{-\beta} \in L^1(B_r(0))$ for some $B_r(0) \Subset \Omega$.
\end{lemma}

\begin{lemma}[Weak Harnack inequality]
\label{Lem3}
Let $r>0$ such that $B_{2r}(0) \subset \Omega$. Assume that $f \geq 0$ and let $v \in X_0^s(\Omega)$, with $v \gneqq 0$ in $\mathbb{R}^N$, be a supersolution to 
\begin{equation*} 
\begin{cases}
(-\Delta )^s v = f & \quad \mathrm{in} \,\, \Omega, \\ v=0 & \quad \mathrm{in} \,\, \big(\mathbb{R}^N \setminus \Omega \big),
\end{cases}
\end{equation*}
i.e.
\begin{equation*} 
\int_{\mathbb{R}^N} (-\Delta)^{\frac{s}{2}} v (-\Delta)^{\frac{s}{2}} \phi \, dx \geq \int_{\Omega} f \phi \, dx,
\end{equation*}
for all non-negative $\phi \in X_0^s(\Omega)$. Then, for every $q < \frac{N}{N-2s}$ there exists a positive constant $C=C(N,s)$ such that
\begin{equation*} 
\Bigg( \int_{B_r(0)} v^q \, dx  \Bigg)^{\frac{1}{q}} \leq C \inf_{B_{\frac{3}{2}r}(0)}v. 
\end{equation*}
\end{lemma}
For the proof of these lemmas see \cite[Lemma 3.10]{MR3479207}, \cite[Theorem 4.10]{MR3479207} and \cite[Theorem 3.4]{MR3479207}, respectively.

In the next two theorems we have our existence results to problem \eqref{Eq1}. At first, we will prove that for $0< \lambda < \Lambda_{N,s}$, and $\gamma \geq 1$ the problem \eqref{Eq1} admits a solution for the case $\mu \in L^1(\Omega)$, and $ f \in L^1(\Omega) \cap X^{-s}(\Omega)$. It is crucial to indicate that our approach in the proof of Theorem \ref{Thm1}, only works for the case $\gamma \geq 1$. However if we further assume that $\mu \in L^{m}(\Omega)$, $m=(\frac{2_s^*}{1-\gamma})'$ ($p'$ denotes the conjugate exponent of $p$) then the same approach works for $\gamma<1$. For a result about the existence with less regularity assumption on $\mu$, see \cite[Theorem 5.3]{MR3479207}. More precisely, the authors showed an existence result for the case $\mu \in L^1(\Omega, |x|^{-(1-\gamma)\beta} \, dx)$.

In the following we denote
\begin{equation*} 
T_n(\sigma)=\begin{cases} \sigma & |\sigma| \leq n \\ n \frac{\sigma}{|\sigma|} & |\sigma| \geq n \end{cases}
\end{equation*}
the usual truncation operator and $G_n(\sigma):=\sigma-T_n(\sigma)$.

\begin{theorem} 
\label{Thm1}
Let $ s \in (0,1)$, $0 < \lambda <\Lambda_{N,s}=\frac {4^s\Gamma^2(\frac{N+2s}4)}{\Gamma^2(\frac{N-2s}4)}$, and $\gamma > 0$. Also assume that $\mu \in L^1(\Omega)$ is a non-negative function and $0 \leq f \in L^1(\Omega) \cap X^{-s}(\Omega)$.
\begin{enumerate}
\item
If $ \gamma=1$, then there is a positive weak solution in $X_0^{s}(\Omega)$ to problem \eqref{Eq1}.
\item
If $\gamma>1$, then there is a positive weak solution in $X_{\mathrm{loc}}^{s}(\Omega)$ to problem \eqref{Eq1} with $T_k^{\frac{\gamma+1}{2}}(u) \in X_0^{s}(\Omega)$ and $G_k(u) \in X_0^{s}(\Omega)$. In addition, if $\frac{4\gamma}{(\gamma+1)^2}>\frac{\lambda}{\Lambda_{N,s}}$, then $u^{\frac{\gamma+1}{2}} \in X_0^{s}(\Omega)$.
\item
If $ \gamma<1$, and furthermore $\mu \in L^{\big(\frac{2_s^*}{1-\gamma}\big)'}(\Omega)$, then there is a positive weak solution in $X_0^{s}(\Omega)$ to problem \eqref{Eq1}.
\end{enumerate}
\end{theorem}

The next theorem gives a necessary and sufficient condition for the existence result to problem \eqref{Eq1}. 

\begin{theorem}[A necessary and sufficient condition for the existence result]
\label{Thm1.5}
Let $ s \in (0,1)$, $0 < \lambda \leq \Lambda_{N,s}$, and $\gamma>0$. Also assume that $0 \leq f, \mu \in L^1(\Omega)$. Then problem \eqref{Eq1} has a positive weak solution if and only if
\begin{equation}
\label{IntegrabilityC}
\int_{\Omega} \frac{f(x)}{|x|^{\beta}} \, dx < + \infty, \qquad \int_{\Omega} \frac{\mu}{\delta^{s (\gamma-1)}} \, dx < + \infty.
\end{equation}
Moreover, the solution $u$ has the following regularity:
\begin{itemize}
\item
$T_k(u) \in X_0^s(\Omega)$ for all $k>0$ and $u \in L^p(\Omega)$ for all $p \in [1, \frac{N}{N-2s})$.
\item
$(-\Delta )^{\frac{s}{2}} u \in L^p(\Omega)$, for all $p \in [1, \frac{N}{N-s})$.
\item 
$u \in X_0^{s_1,p}(\Omega)$, for all $s_1<s$ and for all $p<\frac{N}{N-s}$.
\end{itemize}
\end{theorem}

\begin{remark}
A similar argument as in \cite[Example 3.3]{MR2257147} but with the fractional Laplacian instead of the Laplacian operator shows that problem \eqref{Eq1} does not admit a solution for merely $f \in L^1(\Omega)$.
\end{remark}

The proof of these theorems will appear in the next section. In the following, we will have a non-existence and also a blow-up result for the case that $\lambda >\Lambda_{N,s}$.

The following non-existence result is an immediate consequence of Lemma \ref{Lem1} and Lemma \ref{Lem2}. More precisely, it is well known that the linear problem with Hardy potential has non positive supersolution if $\lambda > \Lambda_{N,s}$. We only bring it here for completeness.
\begin{theorem} 
\label{Thm2}
Let $ s \in (0,1)$, $ \lambda > \Lambda_{N,s}$ and $\gamma >0$. Then there is no positive very weak solution to problem \eqref{Eq1}.
\end{theorem}
\begin{proof}
We argue by contradiction. Let $u$ be a positive very weak solution to problem \eqref{Eq1}. Therefore $u$ satisfies 
\begin{equation*} 
\begin{cases}
(-\Delta )^s u - \Lambda_{N,s} \dfrac{u}{|x|^{2s}} = (\lambda - \Lambda_{N,s}) \dfrac{u}{|x|^{2s}} + g & \quad \mathrm{in} \,\, \Omega, \\ u>0 & \quad \mathrm{in} \,\, \Omega, \\ u=0 & \quad \mathrm{in} \,\, \big(\mathbb{R}^N \setminus \Omega \big),
\end{cases}
\end{equation*}
where $g= \frac{\mu}{u^{\gamma}}+f(x)$. Then by using Lemma \ref{Lem2} and the positivity of $g$ necessarily:
\begin{equation}
\label{Eq3.5}
\Big( (\lambda - \Lambda_{N,s}) \dfrac{u}{|x|^{2s}} \Big) |x|^{-\beta} \in L^1(B_r(0)),
\end{equation}
for some $B_r(0) \Subset \Omega$. On the other hand, by  Lemma \ref{Lem1} we have
\begin{equation}
\label{Eq3.6}
u(x) \geq C |x|^{-\beta}, \qquad \mathrm{in} \,\, B_{r}(0),
\end{equation}
for sufficiently small $r$, where $\beta= \frac{N-2s}{2}-\alpha$ and $\alpha \in [0, \frac{N-2s}{2}) $ is given by the identity
\begin{equation*} 
\frac{{4^s}\Gamma^2(\frac{N+2s}4)}{\Gamma^2(\frac{N-2s}4)} = \frac{4^{s} \Gamma(\frac{N+2s+2\alpha}{4}) \Gamma(\frac{N+2s-2\alpha}{4})}{\Gamma(\frac{N-2s+2\alpha}{4}) \Gamma(\frac{N-2s-2\alpha}{4})}. 
\end{equation*}
The properties of the Gamma function implies $\alpha=0$, see the proof of \cite[Lemma 3.3]{MR3492734}. Now, by combining \eqref{Eq3.5} and \eqref{Eq3.6} we obtain that $|x|^{-N} \in L^1(B_r(0))$, which is a contradiction.
\end{proof}

This non-existence result is strong in the sense that a complete blow-up phenomenon occurs. Namely, if $u_n$ is the solution to the following approximated problem with $\lambda>\Lambda_{N,s}$, where the Hardy potential is substituted by the bounded weight $(|x|^{2s}+\frac{1}{n})^{-1}$, and the singular nonlinearity is substituted by $\frac{\min\{\mu,n\}}{(u_n+\frac{1}{n})^\gamma}$, then $u_n(x_0) \to \infty$, for any $x_0 \in \Omega$, as $n \to \infty$.
\begin{equation}
\label{Eq-Blow-up}
\begin{cases}
(-\Delta )^s u_n = \lambda \dfrac{u_n}{|x|^{2s}+\frac{1}{n}} + \dfrac{\min\{\mu,n\}}{(u_n+\frac{1}{n})^\gamma} + \min\{f,n\} & \quad \mathrm{in} \,\, \Omega, \\ u_n>0 & \quad \mathrm{in} \,\, \Omega, \\ u_n=0 & \quad \mathrm{in} \,\, \big(\mathbb{R}^N \setminus \Omega \big).
\end{cases}
\end{equation}

In the same sprite of Theorem \ref{Thm2}, the proof of this blow-up phenomenon can be obtained taking into consideration that any approximating sequence of non-negative supersolution to the linear problem blow-up in any point of $\Omega$, if $\lambda>\Lambda_{N,s}$, as it is proved in \cite{MR3479207}.

\section{Proof of Theorem \ref{Thm1} and Theorem \ref{Thm1.5}}
\label{Section3}
First of all we prove Theorem \ref{Thm1}. For this purpose let consider the following auxiliary problem: 
\begin{equation}
\label{Eq4}
\begin{cases}
(-\Delta )^s u = \lambda \dfrac{u}{|x|^{2s}} + g & \quad \mathrm{in} \,\, \Omega, \\ u>0 & \quad \mathrm{in} \,\, \Omega, \\ u=0 & \quad \mathrm{in} \,\, \big(\mathbb{R}^N \setminus \Omega \big),
\end{cases}
\end{equation}
where $g \in X^{-s}(\Omega)$. The function $u \in X_0^s(\Omega)$ is a weak energy solution to the above problem if $u \equiv 0$ in $\big(\mathbb{R}^N \setminus \Omega \big)$ and
\begin{equation*} 
\int_{\mathbb{R}^N} (-\Delta)^{\frac{s}{2}} u (-\Delta)^{\frac{s}{2}} \phi \, dx = \lambda \int_{\Omega} \frac{u \phi}{|x|^{2s}} \, dx + \langle g, \phi \rangle_{X^{-s}(\Omega),X_0^s(\Omega)}, \qquad \phi \in X_0^s(\Omega).
\end{equation*}
Here $\langle \cdot, \cdot \rangle_{X^{-s}(\Omega),X_0^s(\Omega)}$ denotes the duality pairing between $X^{-s}(\Omega)$ and $X_0^s(\Omega)$.

The proof of the following Proposition about the existence result for \eqref{Eq4}, can be obtained by using the Hardy inequality and the classical variational methods. See for instance \cite[Section 4.6]{MR3890060}. Also, the uniqueness of the weak energy solution to \eqref{Eq4} follows from the strict monotonicity of the operator $(-\Delta)^{s} u - \lambda \frac{u}{|x|^{2s}}$, for $ 0 \leq \lambda <\Lambda_{N,s}$. The strict monotonicity of this operator is the direct consequence of the Hardy inequality.

\begin{proposition}
\label{Pro1}
If $ g(x) \in L^{2}(\Omega) $, $s \in (0,1)$ and $ 0<\lambda < \Lambda_{N,s} $, then there exists a unique positive weak energy solution to \eqref{Eq4} in $ X_0^s(\Omega)$.
\end{proposition}

Before to continue, we need to define the set $\mathcal{C}$ as the set of functions $v \in L^{2}(\Omega)$ such that there exist positive constants $k_1$ and $k_2$ such that
\begin{equation*} 
k_1\delta^s(x) \leq |x|^{\beta} v(x)  \leq k_2 \delta^s(x),
\end{equation*}
where the constant $\beta$ is given in Lemma \ref{Lem1}, and $ \delta(x)= \mathrm{dist}(x,\partial \Omega)$, $x \in \Omega$, is the distance function from the boundary $\partial \Omega$.

Now, for every $v \in \mathcal{C}$, define $\Phi(v)=w$ where $w \in X_0^s(\Omega)$ is the unique solution to the following problem for any fixed $n$:
\begin{equation}
\label{Eq5}
\begin{cases}
(-\Delta )^s w = \lambda \dfrac{w}{|x|^{2s}} +  \dfrac{\mu_n}{(|v|+\frac{1}{n})^{\gamma}} + f_n(x) & \quad \mathrm{in} \,\, \Omega,\\ w>0 & \quad \mathrm{in} \,\, \Omega, \\ w=0 & \quad \mathrm{in} \,\, \big(\mathbb{R}^N \setminus \Omega \big). 
\end{cases}
\end{equation}
Here $ f_n = T_n(f) $, and $\mu_n = T_n(\mu)$ are the truncations at level $n$.

By Lemma \ref{Lem1}, \cite[Theorem 4.1]{MR3479207} and a result of \cite{MR3168912} it easily follows that $w \in \mathcal{C}$. If we show that $\Phi : \mathcal{C} \to \mathcal{C} $ has a fixed point $w_n$, then $w_n \in \mathcal{C}$ will be the weak solution to the following problem in $X_0^s(\Omega)$.
\begin{equation}
\label{Eq6}
\begin{cases}
(-\Delta )^s w_n = \lambda \dfrac{w_n}{|x|^{2s}} +  \dfrac{\mu_n}{(w_n+\frac{1}{n})^{\gamma}} + f_n(x) & \quad \mathrm{in} \,\, \Omega,\\ w_n>0 & \quad \mathrm{in} \,\, \Omega, \\ w_n=0 & \quad \mathrm{in} \,\, \big(\mathbb{R}^N \setminus \Omega \big). 
\end{cases}
\end{equation}
We apply the Schauder's fixed-point Theorem (see for example \cite[Theorem 3.2.20]{MR3890060}). We need to prove that $\Phi$ is continuous, compact and there exists a bounded convex subset of $ \mathcal{C} \subset L^2(\Omega) $ which is invariant under $\Phi$. 

For continuity let $v_k \to v$ in $L^2(\Omega)$. It is obvious that for each $n$:
\begin{equation*} 
\Bigg\| \Big( \frac{\mu_n}{(|v_k|+\frac{1}{n})^{\gamma}} + f_n\Big)-\Big( \frac{\mu_n}{(|v|+\frac{1}{n})^{\gamma}} + f_n \Big) \Bigg\|_{L^2(\Omega)} \to 0, \qquad k \to \infty.
\end{equation*}
Now, from the uniqueness of the weak solution to \eqref{Eq4}, we conclude $\Phi(v_k) \to \Phi(v)$. 

For compactness, we argue as follows. For $v \in \mathcal{C}$, let $w$ be the solution to \eqref{Eq5}. If $\lambda_1^s(\Omega)$ is the first eigenvalue of $(-\Delta)^s$ in $X_0^s(\Omega)$, \cite[Proposition 9]{MR3002745}, then we have
\begin{equation}
\label{Eq7}
\begin{aligned}
\lambda_1^s(\Omega) \int_{\Omega} w^2 \, dx & \leq \int_{\mathbb{R}^N} | (-\Delta)^{\frac{s}{2}} w|^2 \, dx  \\
& \leq \frac{\Lambda_{N,s}}{\Lambda_{N,s}-\lambda} \, \Bigg( \int_{\mathbb{R}^N} | (-\Delta)^{\frac{s}{2}} w|^2 - \lambda \frac{w^2}{|x|^{2s}} \, dx \Bigg),
\end{aligned}
\end{equation}
where in the last inequality we have used the Hardy inequality. Testing \eqref{Eq5} with $\phi=w$, we have
\begin{equation}
\label{Eq8}
\int_{\mathbb{R}^N} | (-\Delta)^{\frac{s}{2}} w|^2 \, dx - \lambda \int_{\mathbb{R}^N} \frac{w^2}{|x|^{2s}} \, dx = \int_{\Omega} \dfrac{\mu_n}{(|v|+\frac{1}{n})^{\gamma}}  w \, dx + \int_{\Omega} f_n w  \, dx.
\end{equation}
For the first term on the right-hand side of the above equality we have the following estimate:
\begin{align}
\int_{\Omega} \dfrac{\mu_n}{(|v|+\frac{1}{n})^{\gamma}} w \, dx & \leq n^{\gamma} \int_{\Omega} \mu_n w \, dx \leq C_1 \Big( \int_{\Omega} |w|^2 \, dx \Big)^{\frac{1}{2}},  \label{Eq9}
\end{align}
where in the last inequality we have used the H\"older inequality. Once more using H\"older inequality gives $ \int_{\Omega} f_n w \, dx \leq C_2 \Big( \int_{\Omega} |w|^2 \, dx \Big)^{\frac{1}{2}}$ for some $C_2>0$. Thus combining this inequality with \eqref{Eq7}, \eqref{Eq8}, and \eqref{Eq9} we obtain
\begin{equation*} 
\lambda_1^s(\Omega) \int_{\Omega} |w|^2 \, dx \leq C_3 \Big( \int_{\Omega} |w|^2 \, dx \Big)^{\frac{1}{2}},
\end{equation*}
which implies that $\Phi(L^2(\Omega))$ is contained in a ball of finite radius in $L^2(\Omega)$. Therefore the intersection of this ball with $\mathcal{C}$ in invariant under $\Phi$. Moreover, we have $ \int_{\mathbb{R}^N} | (-\Delta)^{\frac{s}{2}} \Phi(v)|^2 \, dx=\int_{\mathbb{R}^N} | (-\Delta)^{\frac{s}{2}} w|^2 \, dx \leq C_4$, which means that $\Phi(L^2(\Omega))$ is relatively compact in $L^2(\Omega)$ by the compactness of the embedding \eqref{Eq3.350}.

\begin{proposition}
\label{Pro1.5}
For every $K \Subset \Omega$, there exists $C_K>0$ such that $\{w_n\}$, the solution to \eqref{Eq6}, satisfies $w_n(x) \geq C_K$ a.e. in $K$, for each $n$.
\end{proposition}
\begin{proof}
Let us consider the following problem:
\begin{equation}
\label{Eq9.1}
\begin{cases}
(-\Delta )^s v_n = \dfrac{\mu_n}{(v_n+\frac{1}{n})^{\gamma}} & \quad \mathrm{in} \,\, \Omega,\\ v_n>0 & \quad \mathrm{in} \,\, \Omega, \\ v_n=0 & \quad \mathrm{in} \,\, \big(\mathbb{R}^N \setminus \Omega \big). 
\end{cases}
\end{equation}
Existence of the weak solution $v_n$ follows from a similar proof to problem \eqref{Eq6}. In the same way of \cite[Lemma 3.2]{MR3356049} we can show that $v_n \leq v_{n+1}$ a.e. in $\Omega$. Also for each $K \Subset \Omega$, there exists $C_K>0$ such that $v_1(x) \geq C_K$ a.e. in $K$. Now subtracting the weak formulation of \eqref{Eq9.1} from the weak formulation of \eqref{Eq6} and using $(w_n-v_n)^-$ as a test function (see \cite[Theorem 20]{MR3393266}) we conclude that $w_n \geq v_n$ a.e. in $\Omega$. Therefore, for every $K \Subset \Omega$, there exists $C_K$ such that 
$w_n \geq v_n \geq v_1 \geq C_k>0$ a.e. in $K$.
\end{proof}

\begin{proposition}
\label{Pro2}
Assume $\gamma \geq 1$. Also let $\{w_n\}_{n=1}^{\infty}$ be the sequence of solutions to \eqref{Eq6}. Then $\{T_k^{\frac{\gamma+1}{2}}(w_n)\}_{n=1}^{\infty}$ and $\{G_k(w_n)\}_{n=1}^{\infty}$ are bounded in $X_0^s(\Omega)$, and $\{T_k(w_n)\}_{n=1}^{\infty}$ is bounded in $X_{\mathrm{loc}}^s(\Omega)$.
\end{proposition} 
\begin{proof}
We will follow the proof of \cite[Theorem 5.2]{MR3479207}. Let $\gamma \geq 1$. Taking $\phi=T_k^{\gamma}(w_n)$ as a test function in \eqref{Eq6} we obtain
\begin{equation}
\label{Eq10}
\begin{aligned}
\int_{\mathbb{R}^N} & (-\Delta )^{\frac{s}{2}} w_n (-\Delta )^{\frac{s}{2}} T_k^{\gamma}(w_n) \, dx = \lambda \int_{\Omega} \dfrac{w_nT_k^{\gamma}(w_n)}{|x|^{2s}} \, dx \\
& \quad + \int_{\Omega} \dfrac{\mu_n}{(w_n+\frac{1}{n})^{\gamma}} T_k^{\gamma}(w_n) \, dx + \int_{\Omega}T_k^{\gamma}(w_n) f_n \, dx.
\end{aligned}
\end{equation}
For the left-hand side, by using \eqref{Eq3.349n} and the following elementary inequality
\begin{equation}
\label{EqElementary}
(s_1-s_2)(s_1^{\gamma}-s_2^{\gamma}) \geq \frac{4\gamma}{(\gamma+1)^2} \Big(s_1^{\frac{\gamma+1}{2}}-s_2^{\frac{\gamma+1}{2}} \Big)^2, \qquad \forall s_1,s_2 \geq 0, \,\, \gamma>0,
\end{equation}
we get
\begin{align} \nonumber 
\int_{\mathbb{R}^N} & (-\Delta )^{\frac{s}{2}} w_n (-\Delta )^{\frac{s}{2}} T_k^{\gamma}(w_n) \, dx \\  \nonumber 
& =  \frac{C_{N,s}}{2} \iint_{D_{\Omega}} \frac{(w_n(x)-w_n(y))(T_k^{\gamma}(w_n)(x)-T_k^{\gamma}(w_n)(y))}{|x-y|^{N+2s}} \, dxdy \\\nonumber 
& \geq \frac{2 \gamma C_{N,s}}{(\gamma+1)^2} \iint_{D_{\Omega}} \frac{|T_k^{\frac{\gamma+1}{2}}w_n(x)-T_k^{\frac{\gamma+1}{2}}(w_n)(y)|^2}{|x-y|^{N+2s}} \, dxdy \\
& \geq C_0 \int_{\mathbb{R}^N} |(-\Delta )^{\frac{s}{2}} T_k^{\frac{\gamma+1}{2}}(w_n)|^2 \, dx. \label{Eq10.0001}
\end{align}
For the first term on the right-hand side we have
\begin{equation}
\label{Eq10.1}
\int_{\Omega} \dfrac{w_n T_k^{\gamma}(w_n)}{|x|^{2s}} \, dx \leq k^{\gamma-1} \int_{\Omega} \dfrac{w_n^2}{|x|^{2s}} \, dx. 
\end{equation}
For the second term on the right-hand side of \eqref{Eq10}, note that 
$\frac{T_k^{\gamma}(w_n)}{(w_n+\frac{1}{n})^{\gamma}} \leq \frac{w_n^{\gamma}}{(w_n+\frac{1}{n})^{\gamma}} \leq 1$.
Now we deduce
\begin{align}
\int_{\Omega} \dfrac{\mu_n}{(w_n+\frac{1}{n})^{\gamma}} T_k^{\gamma}(w_n) \, dx & \leq \int_{\Omega} \mu_n \, dx \leq \|\mu_n\|_{L^1} \leq C_1. \label{Eq13}
\end{align}
Also for the last term:
\begin{align}
\nonumber 
\int_{\Omega}T_k^{\gamma}(w_n) f_n \, dx & \leq k^{\frac{\gamma-1}{2}} \int_{\Omega}T_k^{\frac{\gamma+1}{2}}(w_n) f_n \, dx  \\ \nonumber 
& = k^{\frac{\gamma-1}{2}} \big\langle  f_n, T_k^{\frac{\gamma+1}{2}}(w_n) \big\rangle_{X^{-s}(\Omega),X_0^s(\Omega)} \\ \nonumber 
& \leq k^{\frac{\gamma-1}{2}} \big\langle f, T_k^{\frac{\gamma+1}{2}}(w_n) \big\rangle_{X^{-s}(\Omega),X_0^s(\Omega)} \\  
& \leq k^{\frac{\gamma-1}{2}} \|f\|_{X^{-s}(\Omega)} \|T_k^{\frac{\gamma+1}{2}}(w_n)\|_{X_0^{s}(\Omega)} =k^{\frac{\gamma-1}{2}} C_2 \|T_k^{\frac{\gamma+1}{2}}(w_n)\|_{X_0^{s}(\Omega)}. \label{Eq14}
\end{align}
Thus from \eqref{Eq10}, \eqref{Eq10.0001}, \eqref{Eq10.1}, \eqref{Eq13} and \eqref{Eq14} we obtain
\begin{equation}
\label{Eq15}
\begin{aligned}
\int_{\mathbb{R}^N} |(-\Delta )^{\frac{s}{2}} T_k^{\frac{\gamma+1}{2}}(w_n)|^2 \, dx & \leq \frac{\lambda k^{\gamma-1}}{C_0} \int_{\mathbb{R}^N}   \dfrac{w_n^2}{|x|^{2s}} \, dx  \\
& \quad + C_1+ C_2 k^{\frac{\gamma-1}{2}} \|T_k^{\frac{\gamma+1}{2}}(w_n)\|_{X_0^{s}(\Omega)}.
\end{aligned}
\end{equation}
If we show that the term
\begin{equation}
\label{Eq15.1}
\int_{\Omega} \frac{w_n^2}{|x|^{2s}} \, dx
\end{equation}
is uniformly bounded in $n$, then \eqref{Eq15} gives $ \| T_k^{\frac{\gamma+1}{2}}(w_n)\|_{X_0^s(\Omega)}^2 \leq C_3(k)(1+ \|T_k^{\frac{\gamma+1}{2}}(w_n)\|_{X_0^{s}(\Omega)}) $, which implies the boundedness of $\{T_k^{\frac{\gamma+1}{2}}(w_n)\}$ in $X_0^s(\Omega)$.

For proving the boundedness of \eqref{Eq15.1}, it is enough to consider $\phi=G_k(w_n)$ as a test function in \eqref{Eq6} as follows, where $G_k(\sigma):=\sigma - T_k(\sigma)$.
\begin{equation}
\label{Eq15.2}
\begin{aligned}
\int_{\mathbb{R}^N} |(-\Delta )^{\frac{s}{2}} G_k(w_n)|^2 \, dx \leq \lambda  & \int_{\mathbb{R}^N} \dfrac{w_n G_k(w_n)}{|x|^{2s}} \, dx \\
& + \int_{\Omega} \dfrac{\mu_n}{(w_n+\frac{1}{n})^{\gamma}} G_k(w_n) \, dx + \int_{\Omega} f_n G_k(w_n) \, dx.
\end{aligned}
\end{equation}
Note that for the left-hand side we have used \cite[Proposition 3]{MR3393266}. In order to estimate the terms on the right-hand side of this equality for uniformly in $n$, we have the following.

For the second term on the right-hand side of \eqref{Eq15.2} we have the following estimate uniformly in $n$:
\begin{equation*} 
\int_{\Omega} \dfrac{\mu_n}{(w_n+\frac{1}{n})^{\gamma}} G_k(w_n) \, dx \leq \frac{1}{k^{\gamma-1}} \int_{\Omega} \mu_n \leq C. 
\end{equation*}
For $\int_{\Omega}f_n G_k(w_n)  \, dx$, we have the following estimate: 
\begin{equation*}
\label{Eq15.21}
\begin{aligned}
\int_{\Omega} f_n G_k(w_n)  \, dx = \big\langle  f_n, G_k(w_n) \big\rangle_{X^{-s}(\Omega),X_0^s(\Omega)} & \leq \big\langle  f, G_k(w_n) \big\rangle_{X^{-s}(\Omega),X_0^s(\Omega)} \\
& \leq C_1 \|G_k(w_n)\|_{X_0^s(\Omega)}. 
\end{aligned}
\end{equation*}
For the first term on the right-hand side of \eqref{Eq15.2} we can write:
\begin{equation}
\label{Eq15.5}
\int_{\mathbb{R}^N} \dfrac{w_nG_k(w_n)}{|x|^{2s}} \, dx = \int_{\mathbb{R}^N} \dfrac{|G_k(w_n)|^2}{|x|^{2s}} \, dx + k \int_{\mathbb{R}^N} \frac{G_k(w_n)}{|x|^{2s}} \, dx.
\end{equation}
For the last term in \eqref{Eq15.5}, by using the H\"older inequality with exponents $a=2_s^*$ and $b=\frac{2N}{N+2s}<\frac{N}{2s}$, noting that the integration can be over $\Omega$, because of $w_n \equiv 0$ in $\mathbb{R}^N \setminus \Omega$, and the embedding \eqref{Eq3.350} we obtain

\begin{equation*}
\label{Eq15.5addeda}
\begin{aligned}
 \int_{\mathbb{R}^N} \frac{G_k(w_n)}{|x|^{2s}} \, dx = \int_{\Omega} \frac{G_k(w_n)}{|x|^{2s}} \, dx & \leq \Bigg( \int_{\Omega} \frac{1}{|x|^{2sb}} \, dx \Bigg)^{\frac{1}{b}} \Bigg( \int_{\mathbb{R}^N} |G_k(w_n)|^a \, dx \Bigg)^{\frac{1}{a}} \\
&  \leq C_2 \| G_k(w_n)\|_{X_0^s(\Omega)}. 
\end{aligned}
\end{equation*}
Combining the above estimates, from \eqref{Eq15.2} we get
\begin{equation*} 
\begin{aligned}
\int_{\mathbb{R}^N} |(-\Delta )^{\frac{s}{2}} G_k(w_n)|^2 \, dx - \lambda \int_{\mathbb{R}^N} \dfrac{|G_k(w_n)|^2}{|x|^{2s}} \, dx & \leq kC_2 \| G_k(w_n)\|_{X_0^s(\Omega)} \\
& \quad + C+C_1 \| G_k(w_n)\|_{X_0^s(\Omega)}.
\end{aligned}
\end{equation*}
Now Hardy inequality shows the boundedness of the term, $ \int_{\mathbb{R}^N} \dfrac{|G_k(w_n)|^2}{|x|^{2s}} \, dx $, and therefore we obtain the boundedness of \eqref{Eq15.1} by using the fact that $w_n^2 \leq 2 ( T_k^2(w_n)+G_k^2(w_n))$, i.e.
\begin{equation*} 
\begin{aligned}
\int_{\Omega} \frac{w_n^2}{|x|^{2s}} \, dx & \leq 2 \int_{\Omega} \frac{|T_k(w_n)|^2}{|x|^{2s}} \, dx + 2 \int_{\Omega} \frac{|G_k(w_n)|^2}{|x|^{2s}} \, dx \\
& \leq 2k^2 \int_{\Omega} \frac{1}{|x|^{2s}} \, dx + 2 \int_{\Omega} \frac{|G_k(w_n)|^2}{|x|^{2s}} \, dx.
\end{aligned}
\end{equation*}
Moreover, we get the boundedness of $\| G_k(w_n)\|_{X_0^s(\Omega)}$ uniformly in $n$. 

Now we show that $\{T_k(w_n)\}$ is bounded in $X_{\mathrm{loc}}^s(\Omega)$. For this purpose first note that by Proposition \ref{Pro1.5}, for any compact set $K \Subset \Omega$, there exists $C(K)>0$ such that
\begin{equation*}  
w_n(x) \geq w_1(x) \geq C(K) >0, \qquad \text{a.e. in} \,\, K.
\end{equation*}
Therefore
\begin{equation*} 
T_k(w_n) \geq T_k(w_1) \geq \tilde{C}:=\min\{k, C(K) \}.
\end{equation*}
For $(x,y) \in K \times K$, define $\alpha_n:= \dfrac{T_k(w_n)(x)}{\tilde{C}}$, and $\beta_n:= \dfrac{T_k(w_n)(y)}{\tilde{C}}$.
Since $\alpha_n, \beta_n \geq 1$, we have the following estimate by applying an elementary inequality
\begin{equation*} 
(\alpha_n - \beta_n)^2 \leq \Big(\alpha_n^{\frac{\gamma+1}{2}}-\beta_n^{\frac{\gamma+1}{2}} \Big)^2. 
\end{equation*}
Now by the definition of $\alpha_n$ and $\beta_n$, we obtain
\begin{equation*} 
\Big(T_k(w_n(x)) - T_k(w_n(y)) \Big)^2 \leq \tilde{C}^{1-\gamma} \Big(T_k^{\frac{\gamma+1}{2}}w_n(x)-T_k^{\frac{\gamma+1}{2}}w_n(y) \Big)^2.
\end{equation*}
Thus we get the boundedness of $\{T_k(w_n)\}_{n=1}^{\infty}$ in $X_{\mathrm{loc}}^s(\Omega)$ by using \eqref{EqEqEqOO.} and the boundedness of $\{T_k^{\frac{\gamma+1}{2}}(w_n)\}_{n=1}^{\infty}$ in $X_0^s(\Omega)$.
\end{proof}

\begin{remark}
In the case $\gamma=1$, since both $\{G_k(w_n)\}_{n=1}^{\infty}$ and $\{T_k(w_n)\}_{n=1}^{\infty}$ are bounded in $X_0^s(\Omega)$, therefore $\{w_n\}_{n=1}^{\infty}$ is bounded in $X_0^s(\Omega)$.
\end{remark}

\begin{remark}
For the case $0<\gamma < 1$, if furthermore we assume $\mu \in L^{\big(\frac{2_s^*}{1-\gamma}\big)'}(\Omega)$, then the sequence $\{w_n\}_{n=1}^{\infty}$ is bounded in $X_0^s(\Omega)$. Indeed, you just have to keep in mind that
\begin{equation*} 
\int_{\Omega} \dfrac{\mu_n}{(w_n+\frac{1}{n})^{\gamma}} w_n \, dx \leq \int_{\Omega} \mu_n w_n^{1-\gamma} \, dx \leq \|\mu\|_{L^{\big(\frac{2_s^*}{1-\gamma}\big)'}(\Omega) } \| w_n \|_{L^{2_s^*}(\Omega)}^{1-\gamma} \leq C \|w_n\|_{X_0^s(\Omega)}^{1-\gamma}.
\end{equation*}
Since the rest of the proof can be obtained proceeding as in the case $\gamma=1$, for the sake of brevity it is left to the reader.
\end{remark}

Now we are ready to proof Theorem \ref{Thm1}.
\begin{proof}[Proof of Theorem \ref{Thm1}]
There exists $ u \in X_{\mathrm{loc}}^{s}(\Omega)$ ($u \in X_0^s(\Omega)$ in the case $\gamma \leq 1$) such that up to a subsequence 
\begin{itemize}
\item
$ w_n \to u$ weakly in $X_{\mathrm{loc}}^{s}(\Omega)$ (weakly in $X_0^s(\Omega)$ in the case $\gamma \leq 1$).
\item
$ G_k(w_n) \to G_k(u)$ weakly in $X_0^{s}(\Omega)$.
\item
$T_k^{\frac{\gamma+1}{2}}(w_n) \to T_k^{\frac{\gamma+1}{2}}(u) $ weakly in $X_0^{s}(\Omega)$.
\end{itemize}
Also, by using the embedding \eqref{Eq3.350}, up to a subsequence we may have
\begin{itemize}
\item
$ w_n \to u $ in $L^r(\Omega)$, for any $r \in [1,2_{s}^*)$.
\item
$w_n(x) \to u(x)$ pointwise a.e. in $\Omega$.
\end{itemize} 
Now for every fixed $\phi \in \mathcal{T}(\Omega)$, by the estimates above, we could pass to the limit and obtain
\begin{equation*} 
\begin{aligned}
& \int_{\Omega} \frac{w_n \phi}{|x|^{2s}} \, dx \to \int_{\Omega} \frac{u \phi}{|x|^{2s}} \, dx < + \infty \\
& \int_{\Omega} \dfrac{\mu_n}{(w_n+\frac{1}{n})^{\gamma}} \phi \, dx \to \int_{\Omega} \dfrac{\mu \phi}{u^{\gamma}} \, dx< + \infty \\
& \int_{\Omega} f_n \phi  \, dx \to \int_{\Omega} f \phi \, dx.
\end{aligned}
\end{equation*}
Also, for every $\phi \in \mathcal{T}(\Omega)$, we have
\begin{equation*}  
\lim_{n \to \infty} \int_{\mathbb{R}^N} (-\Delta )^{\frac{s}{2}} w_n (-\Delta )^{\frac{s}{2}} \phi \, dx =  \lim_{n \to \infty} \int_{\mathbb{R}^N} w_n (-\Delta )^{s} \phi \, dx = \int_{\mathbb{R}^N} u (-\Delta)^{s} \phi \, dx.
\end{equation*}
Since for every $K \Subset \Omega$, there exists $C_K>0$ such that $w_n(x) \geq C_K$ a.e. in $K$ and also $w_n \equiv 0$ in $\big(\mathbb{R}^N \setminus \Omega \big)$ and because of $w_n(x) \to u(x)$ a.e. in $\Omega$, thus $u$ is a weak solution to problem \eqref{Eq1}.

Finally note that if we take $\gamma$ such that $\frac{4\gamma}{(\gamma+1)^2}>\frac{\lambda}{\Lambda_{N,s}}$, then by testing \eqref{Eq6} with $w_n^{\gamma}$, and using the inequality \eqref{EqElementary} together with Hardy inequality, it easily follows that $u^{\frac{\gamma+1}{2}} \in X_0^{s}(\Omega)$.
\end{proof}

By now, in Theorem \ref{Thm2} we have shown that for $\lambda> \Lambda_{N,s}$ there is no positive solution to problem \eqref{Eq1}. Also, in Theorem \ref{Thm1} we have proved the existence of a positive solution for $\lambda< \Lambda_{N,s}$. The following remark for $\lambda=\Lambda_{N,s}$ may be interesting.

\begin{remark}
In the borderline case $\lambda=\Lambda_{N,s}$, by invoking the improved version of Hardy inequality, \cite{MR3186917}, one can define the space $H(\Omega)$ as the completion of $C_0^{\infty}(\Omega)$ with respect to the norm:
\begin{equation*} 
\|\phi\|_{H(\Omega)} := \Bigg( \int_{\mathbb{R}^N} |(-\Delta)^{\frac{s}{2}} \phi|^2 \, dx - \Lambda_{N,s} \int_{\Omega} \frac{\phi^2}{|x|^{2s}} \, dx \Bigg)^{\frac{1}{2}}.
\end{equation*}
It can be proved that $X_0^s(\Omega) \subsetneq H(\Omega) \subsetneq X_0^{s,q}(\Omega)$, for all $q <2$. By invoking the classical variational methods in the space $H(\Omega)$ and the same techniques used above, a similar existence result can be obtained in this new function space. See \cite[Remark 1]{MR3186917} and also \cite{MR3231059} for the details.
\end{remark}

Now, in the spirit of \cite[Theorem 4.10]{MR3479207} we prove Theorem \ref{Thm1.5} which gives a necessary and sufficient condition for the existence of a solution to \eqref{Eq1}. 
\begin{proof}[Proof of Theorem \ref{Thm1.5}]
Let consider $u$ as a weak solution to problem \eqref{Eq1} and $\phi_n \in \mathcal{T}(\Omega)$ as the weak energy solutions to the following problems:
\begin{equation*} 
\begin{cases}
(-\Delta )^s \phi_n = \lambda \dfrac{\phi_{n-1}}{|x|^{2s}+\frac{1}{n}}+1 & \quad \mathrm{in} \,\, \Omega,\\ \phi_n>0 & \quad \mathrm{in} \,\, \Omega, \\ \phi_n=0 & \quad \mathrm{in} \,\, \big(\mathbb{R}^N \setminus \Omega \big),
\end{cases}
\end{equation*}
where the iteration starts with
\begin{equation*} 
\begin{cases}
(-\Delta )^s \phi_0 = 1 & \quad \mathrm{in} \,\, \Omega,\\ \phi_0>0 & \quad \mathrm{in} \,\, \Omega, \\ \phi_0=0 & \quad \mathrm{in} \,\, \big(\mathbb{R}^N \setminus \Omega \big).
\end{cases}
\end{equation*}
The Comparison Principle for fractional Laplacian operator implies that $\phi_0 \leq \phi_1 \leq \cdots \leq \phi_{n-1} \leq \phi_n \leq \phi$, where $\phi:=\lim_{n \to \infty} \phi_n$ is the weak energy solution to 
\begin{equation}
\label{EqNS0}
\begin{cases}
(-\Delta )^s \phi = \lambda \dfrac{\phi}{|x|^{2s}}+1 & \quad \mathrm{in} \,\, \Omega,\\ \phi>0 & \quad \mathrm{in} \,\, \Omega, \\ \phi=0 & \quad \mathrm{in} \,\, \big(\mathbb{R}^N \setminus \Omega \big).
\end{cases}
\end{equation}
Using $\phi_n$ as a test function in \eqref{Eq1} implies that 
\begin{equation}
\label{EqNS1}
\int_{\mathbb{R}^N}  u (-\Delta)^{s} \phi_n \, dx = \lambda \int_{\Omega} \frac{u \phi_n}{|x|^{2s}} \, dx + \int_{\Omega} \frac{\mu\phi_n }{u^\gamma} \, dx  + \int_{\Omega} f \phi_n \,dx.
\end{equation}
On the other hand, by the definition $\phi_n$, we have 
\begin{equation}
\label{EqNS2}
\int_{\mathbb{R}^N}  u (-\Delta)^{s} \phi_n \, dx = \lambda \int_{\Omega} \frac{u \phi_{n-1}}{|x|^{2s}+\frac{1}{n}} \, dx + \int_{\Omega} u \,dx.
\end{equation}
Combining \eqref{EqNS1} and \eqref{EqNS2} and noticing that $\frac{\phi_{n-1}}{|x|^{2s}+\frac{1}{n}} \leq \frac{\phi_n}{|x|^{2s}}$, we get
\begin{equation*} 
\int_{\Omega} f \phi_n \, dx \leq \int_{\Omega}u \, dx = C. 
\end{equation*}
Therefore, the sequence $\{f \phi_n\}$ is uniformly bounded in $L^1(\Omega)$. Also, since $\{f \phi_n\}$ is increasing, applying the Monotone Convergence Theorem and invoking Lemma \ref{Lem1} we obtain
\begin{equation*}
C_1 \int_{B_r(0)} |x|^{-\beta} f \, dx \leq \int_{\Omega} f \phi \, dx \leq C.
\end{equation*}
Also, from Remark \ref{RemarkADD01}, it follows that 
\begin{equation*} 
\int_{\Omega} \frac{\mu}{\delta^{s(\gamma-1)}} \, dx < + \infty. 
\end{equation*}

Now assume that
\begin{equation}
\label{EqSSUU1}
\int_{B_r(0)} |x|^{-\beta} f \, dx \leq C, \qquad \text{for some $r$ and} \,\, B_r(0) \Subset \Omega,
\end{equation}
and
\begin{equation}
\label{EqSSUU+1}
\int_{\Omega} \frac{\mu}{\delta^{s(\gamma-1)}} \, dx < + \infty.
\end{equation}
Let $u_n \in X_0^s(\Omega)$ be the weak energy solutions to the problems
\begin{equation}
\label{EqNS3}
\begin{cases}
(-\Delta )^s u_n = \lambda \dfrac{u_{n-1}}{|x|^{2s}+\frac{1}{n}}+\dfrac{\mu}{(u_{n-1}+\frac{1}{n})^\gamma}+f_n & \quad \mathrm{in} \,\, \Omega,\\ u_n>0 & \quad \mathrm{in} \,\, \Omega, \\ u_n=0 & \quad \mathrm{in} \,\, \big(\mathbb{R}^N \setminus \Omega \big),
\end{cases}
\end{equation}
where
\begin{equation*} 
\begin{cases}
(-\Delta )^s u_0 = f_1 & \quad \mathrm{in} \,\, \Omega,\\ u_0>0 & \quad \mathrm{in} \,\, \Omega, \\ u_0=0 & \quad \mathrm{in} \,\, \big(\mathbb{R}^N \setminus \Omega \big).
\end{cases}
\end{equation*}
Here $f_n=T_n(f)$. Again we have $u_0 \leq u_1 \leq \cdots \leq u_{n-1} \leq u_n$ in $\mathbb{R}^N$. Using $\phi \in X_0^s(\Omega)$, the solution to \eqref{EqNS0}, as a test function in \eqref{EqNS3} we obtain
\begin{equation}
\label{EqNS4}
\int_{\mathbb{R}^N}  u_n (-\Delta)^{s} \phi \, dx = \lambda \int_{\Omega} \frac{u_{n-1} \phi}{|x|^{2s}+\frac{1}{n}} \, dx + \int_{\Omega} \frac{\mu\phi}{(u_{n-1}+\frac{1}{n})^\gamma} \, dx  + \int_{\Omega} f_n \phi \,dx.
\end{equation}
On the other hand, using $u_n$ as a test function in the weak formulation of \eqref{EqNS0}, we get
\begin{equation}
\label{EqNS5}
\int_{\mathbb{R}^N}  u_n (-\Delta)^{s} \phi \, dx = \lambda \int_{\Omega} \frac{u_{n} \phi}{|x|^{2s}} \, dx + \int_{\Omega} u_n \,dx.
\end{equation}
From \eqref{EqNS4} and \eqref{EqNS5} and using Lemma \ref{Lem1} together with \eqref{EqSSUU1}, and \eqref{EqSSUU+1} we obtain
\begin{equation}
\label{EqNS5.0005}
\begin{aligned}
\int_{\Omega} u_n \, dx \leq \int_{\Omega} f_n \phi \, dx + &  \int_{\Omega} \frac{\mu \phi}{(u_{n-1}+\frac{1}{n})^\gamma} \, dx \leq \int_{\Omega} f \phi \, dx + \int_{\Omega} \frac{\mu \phi}{u_{0}^{\gamma}} \, dx \\
& \leq C_1 \int_{\Omega} f |x|^{-\beta} \, dx + c_1 c^{-\gamma} \int_{\Omega} \frac{\mu}{\delta^{s(\gamma-1)}} \, dx \\
& \leq C.
\end{aligned}
\end{equation}
Notice that in the last inequality, we have used $u_0 \geq c \delta^s$, and  $\phi \sim c_1\delta^s$ near the boundary, $\partial \Omega$, for some $c_1>0$, since $\phi$ is the solution to \eqref{EqNS0}. This follows by a result of \cite{MR3168912} together with the Comparison Principle for the fractional Laplacian.

Since $u_n$ is increasing and also uniformly bounded in $L^1(\Omega)$, by the Monotone Convergence Theorem we conclude that $ u := \lim_{n \to \infty} u_n$ is a function in $L^1(\Omega)$. We want to show that $u$ is a weak solution to problem \eqref{Eq1}. For this purpose let $\psi \in X_0^s(\Omega) \cap L^{\infty}(\Omega)$ be the unique positive weak energy solution to  
\begin{equation*} 
\begin{cases}
(-\Delta )^s \psi = 1 & \quad \mathrm{in} \,\, \Omega, \\ \psi=0 & \quad \mathrm{in} \,\, \big(\mathbb{R}^N \setminus \Omega \big).
\end{cases}
\end{equation*}
Using $\psi$ as a test function in \eqref{EqNS3} and noting that $\psi \sim \delta^s$, from \eqref{EqNS5.0005} we get
\begin{equation*}  \lambda \int_{\Omega} \frac{u_{n-1}}{|x|^{2s}+\frac{1}{n}} \,  \delta^s dx + \int_{\Omega} \frac{\mu }{(u_{n-1}+\frac{1}{n})^\gamma} \, \delta^s dx \leq C_2 \int_{\Omega} u_n \, dx \leq C_2C.
\end{equation*}
Thus by applying the Monotone Convergence Theorem we get 
\begin{equation*} 
\dfrac{u_{n-1}}{|x|^{2s}+\frac{1}{n}} + f_n \nearrow \dfrac{u}{|x|^{2s}} + f, \qquad \mathrm{in\,\,} L^1(\Omega, \, \delta^s dx).
\end{equation*}
Also since 
\begin{equation*}  
\Big|\frac{\mu}{(u_{n-1}+\frac{1}{n})^\gamma} \, \delta^s \Big| \leq \Big|\frac{\mu}{u_{0}^\gamma} \, \delta^s \Big| \leq \frac{\mu}{\delta^{s(\gamma-1)}} \in L^1(\Omega), 
\end{equation*} 
by the Dominated Convergence Theorem we have
\begin{equation*} 
\frac{\mu}{(u_{n-1}+\frac{1}{n})^\gamma} \to \frac{\mu}{u^\gamma}, \qquad \mathrm{in\,\,} L^1(\Omega, \, \delta^s dx). \end{equation*}
Therefore, $u$ satisfies the equation \eqref{Eq1} in the following weak sense:
\begin{equation*} 
\int_{\mathbb{R}^N} u (-\Delta)^{s} \phi \, dx = \lambda \int_{\Omega} \frac{u \phi}{|x|^{2s}} \, dx + \int_{\Omega} \frac{\mu \phi}{u^\gamma} \, dx  + \int_{\Omega} f \phi \,dx, \quad \forall \phi \in \mathcal{T}(\Omega).
\end{equation*}
Tesing $T_k(u_n)$ in \eqref{EqNS3}, and using \eqref{EqSSUU+1}, we can show that $T_k(u_n) \to T_k(u)$ weakly in $X_0^s(\Omega)$ (similar to the arguments in the proof of Proposition \ref{Pro2}).  Moreover, since $ \lambda \dfrac{u_{n-1}}{|x|^{2s}+\frac{1}{n}}+\dfrac{\mu}{(u_{n-1}+\frac{1}{n})^\gamma}+f_n $ converges strongly in $L^1(\Omega, \, \delta^s dx)$, then by mimicking the proofs of \cite[Proposition 2.3]{GVPF} and \cite[Theorem 23]{MR3393266} (or directly by adapting the Green operator’s viewpoint of the Laplacian case \cite[Theorem 1.2.2]{MR3156649}), we obtain
\begin{itemize}
\item $ u \in L^p(\Omega)$ for all $ p \in [1, \frac{N}{N-2s})$.
\item $(-\Delta )^{\frac{s}{2}} u \in L^p(\Omega)$, for all $p \in [1, \frac{N}{N-s})$.
\end{itemize}
Noting that since we have $N>2s$, therefore $\frac{N}{N-s} <2$. Now, by invoking Theorem 5 and Proposition 10 in chapter 5 of the reference book \cite{MR0290095}, we get that $u \in X_0^{s_1,p}(\Omega)$, for all $s_1<s$ and for all $p<\frac{N}{N-s}$. (In \cite{MR0290095}, $X_0^{s,p}(\Omega)$ reads as $\Lambda_{s}^{p,p}(\mathbb{R}^N)$, and $\mathcal{L}_{s}^p(\mathbb{R}^N)$ denotes the space of Bessel potentials, see \cite[subsection 3.2]{MR0290095}.)

\end{proof}

\section{Some uniqueness results and the rate of the growth of solutions}
\label{Section3.5}

In this section, we have some uniqueness results. Also, with some summability assumptions on the data of $\mu$ and $f$, we find the rate of the growth of solutions. 

At first for the special case $\mu \equiv 1$, by studying the behaviour of solutions near the boundary we discuss the uniqueness of solutions to problem \eqref{Eq1}.

\begin{proposition} 
\label{PropositionEs}
If $\mu \equiv 1$ then the solution obtained to problem \eqref{Eq1} in Theorem \ref{Thm1} behaves as:
\begin{equation}
\label{PropositionEsEq}
\begin{cases}
k_1 \delta^{s} (x) \leq |x|^{\beta} u(x), \qquad & 0<\gamma<1, \\
k_1 \delta^{s} (x) \left( \ln \left(\dfrac{r}{\delta^s (x)} \right) \right)^{\frac{1}{2}} \leq |x|^{\beta} u(x), \qquad & \gamma=1,   \\
k_1 \delta^{\frac{2s}{\gamma+1}} (x) \leq |x|^{\beta} u(x), \qquad & \gamma>1,
\end{cases}
\end{equation}
for any $x \in \Omega$, and some $k_1>0$, where $r> \mathrm{diam}(\Omega)$. Here $\beta$ is as defined in Lemma \ref{Lem1}.
\end{proposition}

\begin{proof}
First of all notice that by Lemma \ref{Lem1}, there exist a constant $C_1>0$ such that
\begin{equation}
\label{ESTIMMM1}
|x|^{\beta} u(x) \geq C_1, \qquad \text{in} \,\, B_{\epsilon}(0).
\end{equation}
Now let $w$ be the weak energy solution to the following problem.
\begin{equation*} 
\begin{cases}
(-\Delta )^s w = \dfrac{1}{w^{\gamma}} & \quad \mathrm{in} \,\, \Omega,\\ w>0 & \quad \mathrm{in} \,\, \Omega, \\ w=0 & \quad \mathrm{in} \,\, \big(\mathbb{R}^N \setminus \Omega \big).
\end{cases}
\end{equation*}
By \cite[Theorem 2.9]{MR3842325} or \cite[Theorem 1.2]{MR3797614} we know that $w$ satisfies:
\begin{equation}
\label{ESTIMMM2}
\begin{cases}
k_1 \delta^{s} (x) \leq w(x) \leq k_2 \delta^{s} (x), \qquad & 0<\gamma<1, \\
k_1 \delta^{s} (x) \left( \ln \left(\dfrac{r}{\delta^s (x)} \right) \right)^{\frac{1}{2}} \leq w(x) \leq k_1 \delta^{s} (x) \left( \ln \left(\dfrac{r}{\delta^s (x)} \right) \right)^{\frac{1}{2}}, \qquad & \gamma=1, \\
k_1 \delta^{\frac{2s}{\gamma+1}} (x) \leq w(x) \leq k_2 \delta^{\frac{2s}{\gamma+1}} (x), \qquad & \gamma>1,
\end{cases}
\end{equation} 
for some $k_1,k_2>0$, $r>\mathrm{diam}(\Omega)$, and any $x \in \Omega$. By the Comparison Principle for the fractional Laplacian operator, e.g. \cite[Proposition 2.17]{MR2270163}, we obtain $u(x) \geq w(x)$, which together with \eqref{ESTIMMM1} and \eqref{ESTIMMM2} gives \eqref{PropositionEsEq}.
\end{proof}

\begin{remark}
\label{Reemark}
Notice that by using the estimates in Proposition \ref{PropositionEs} and applying the H\"older inequality and the fractional Hardy-Sobolev inequality (and convexity of $\Omega$ only for $0<s<\frac{1}{2}$), \cite[Theorem 1.1]{MR3814363}, we get

\begin{itemize}
\item
If $0<\gamma<1$:
\begin{equation*} 
\begin{aligned}
\Bigg| \int_{\Omega} \frac{\phi}{u^{\gamma}} \, dx \Bigg| \leq  k_1^{-\gamma} \int_{\Omega} \frac{|x|^{\beta \gamma} |\phi|}{\delta^{s\gamma}} \, dx \leq C \Bigg( \int_{\Omega} \frac{\phi^2}{\delta^{2s\gamma}} \, dx \Bigg)^{\frac{1}{2}} & \leq C_1 \| \phi \|_{X_0^{s\gamma}(\Omega)} \\
& \leq C_2 \| \phi \|_{X_0^{s}(\Omega)},
\end{aligned}
\end{equation*}
where in the last inequality, we used the continuous embedding of $X_0^{s_2}(\Omega)$ into $X_0^{s_1}(\Omega)$, for any $s_1<s_2$.
\item
If $\gamma=1$:
\begin{equation*} 
\begin{aligned}
\Bigg| \int_{\Omega} \frac{\phi}{u} \, dx \Bigg| & \leq  k_1^{-1} \int_{\Omega} \frac{|x|^{\beta}|\phi|}{\delta^s(x) \left( \ln \left(\dfrac{r}{\delta^s (x)} \right) \right)^{\frac{1}{2}}} \, dx \\
& \leq C \Bigg( \int_{\Omega}  \frac{1}{\left|\ln \left(\dfrac{r}{\delta^s (x)} \right) \right|} \, dx \Bigg)^{\frac{1}{2}} \Bigg( \int_{\Omega} \frac{\phi^2}{\delta^{2s}} \, dx \Bigg)^{\frac{1}{2}} & \leq C_1 \| \phi \|_{X_0^{s}(\Omega)}.
\end{aligned}
\end{equation*}

\item
If $\gamma>1$:
\begin{equation*} 
\begin{aligned}
\Bigg| \int_{\Omega} \frac{\phi}{u^{\gamma}} \, dx \Bigg| \leq  k_1^{-\gamma} \int_{\Omega} \frac{|x|^{\beta \gamma} |\phi|}{\delta^{\frac{2s \gamma}{\gamma+1}}} \, dx & \leq C \Bigg( \int_{\Omega} \frac{1}{\delta^{2s\frac{\gamma-1}{\gamma+1}}} \, dx \Bigg)^{\frac{1}{2}} \Bigg( \int_{\Omega} \frac{\phi^2}{\delta^{2s}} \, dx \Bigg)^{\frac{1}{2}} \\
& \leq C_1 \Bigg( \int_{\Omega} \frac{1}{\delta^{2s\frac{\gamma-1}{\gamma+1}}} \, dx \Bigg)^{\frac{1}{2}} \|\phi\|_{X_0^{s}(\Omega)}.
\end{aligned}
\end{equation*}
If in addition we assume $2s(\gamma-1)<\gamma+1$, then
\begin{equation*} 
\int_{\Omega} \frac{\phi}{u^{\gamma}} \, dx \leq C_2 \|\phi\|_{X_0^{s}(\Omega)}.
\end{equation*}
\end{itemize}

For general domains with some boundary regularity, the fractional Hardy-Sobolev inequality is proved for $s \in[ \frac{1}{2},1)$. See \cite{MR3933834, MR2085428, MR3021545}. But in \cite{MR3814363}, the authors proved the fractional Hardy-Sobolev inequality for any $s \in (0,1)$, by using the fact that the domain is a convex set and its distance from the boundary is a superharmonic function.
\end{remark}

\subsection*{Uniqueness in the special case $\mu \equiv 1$, and $0<\gamma \leq 1$, or $\gamma>1$ with $2s(\gamma-1)<\gamma+1$.}
Let $u_1$ and $u_2$ be two solutions in $X_{\mathrm{loc}}^s(\Omega)$ to problem \eqref{Eq1} and define $w=u_1-u_2$. Then we have
\begin{equation}
\label{Uniquness11}
\int_{\mathbb{R}^N} w (-\Delta)^{s} \phi \, dx = \lambda \int_{\Omega} \frac{w \phi}{|x|^{2s}} \, dx + \int_{\Omega} \frac{\phi}{u_1^{\gamma}}-\frac{\phi}{u_2^{\gamma}} \, dx, \quad \forall \phi \in \mathcal{T}(\Omega).
\end{equation}
The fractional Hardy-Sobolev inequality and a density argument, shows that the equality \eqref{Uniquness11} holds for all $\phi \in X_{0}^s(\Omega)$, see remark \ref{Reemark}. This means that $w \in X_0^s(\Omega)$. Now by using $w^-$ as a test function in \eqref{Uniquness11} and applying Hardy inequality we deduce that $w^-\equiv 0$. So we reach at the conclusion that $u_1 \geq u_2$. Similar argument shows that $u_1 \leq u_2$. Therefore $u_1=u_2$, and the uniqueness follows.

\begin{remark}
The assumption $\mu \equiv 1$ is taken for the purpose of simplification. However, we can assume any $\mu \geq m$, for some positive constant $m$, such that
\begin{equation*} 
\begin{cases}
\displaystyle\int_{\Omega} \mu^2 \delta^{2s(1-\gamma)} \, dx < + \infty \qquad & 0<\gamma<1, \\
\displaystyle\int_{\Omega}  \dfrac{\mu^2}{ \left| \ln \left( \dfrac{r}{\delta^s} \right) \right|} \, dx < + \infty \qquad & \gamma=1, \\
\displaystyle\int_{\Omega} \mu^2 \delta^{2s \frac{1-\gamma}{1+\gamma}} \, dx < + \infty \qquad & \gamma>1, \,\,\, \text{and} \,\,\, \gamma(2s-1)<(2s+1), \\
\end{cases}
\end{equation*}
and the above argument works. For a further discussion see \cite[Theorem 5.2]{MR3842325} which is about a Brezis-Oswald type result concerning uniqueness.
\end{remark}

Once again, because of the interest in uniqueness, we have another definition to solutions of \eqref{Eq1}. In fact, we would like to consider the entropy solution. The motivation of the definition comes from the works \cite{MR1354907, MR1409661}. In what follows, we would consider $0< \gamma \leq 1$.
\begin{definition}
\label{NewDEFN}
Assume $0 \leq \mu, f \in L^1(\Omega)$, and $0< \gamma \leq 1$. We say that $u$ is an entropy solution to \eqref{Eq1} if
\begin{itemize} 
\item
for every $K \Subset \Omega$, there exists $C_K>0$ such that $u(x) \geq C_K$ in $K$ and also $u \equiv 0$ in $\big(\mathbb{R}^N \setminus \Omega \big)$;
\item
$T_k(u) \in X_0^s(\Omega)$, for every $k$, and $u$ satisfies the following family of inequalities:
\begin{equation*}
\label{Eq3.500}
\begin{aligned}
\int_{\{|u-\phi| < k \}} (-\Delta)^{\frac{s}{2}} u (-\Delta)^{\frac{s}{2}} (u-\phi) \, dx \leq \lambda & \int_{\Omega} \frac{u T_k(u-\phi)}{|x|^{2s}} \, dx + \int_{\Omega} u^{-\gamma} \mu T_k(u-\phi) \, dx \\
& + \int_{\Omega} fT_k(u-\phi) \, dx,
\end{aligned}
\end{equation*}
for any $k$ and any $\phi \in X_0^s(\Omega) \cap L^{\infty}(\Omega)$, and also together with this extra assumption that the second term on the right-hand side of the above inequality be finite for any $\phi \in X_0^s(\Omega) \cap L^{\infty}(\Omega)$. The well-posedness of this term will be clear after the construction of entropy solution.
\end{itemize}
\end{definition}

Let $u$ and $v$ be two entropy solution. Testing $u$ with $\phi=T_h(v)$ and $v$ with $\phi=T_h(u)$ in the weak formulation of entropy inequalities, we have 
\begin{equation}
\label{ENTEq20}
\begin{aligned}
\int_{\{|u-T_h(v)|<k\}} & (-\Delta)^{\frac{s}{2}} u (-\Delta)^{\frac{s}{2}} (u-T_h(v)) \, dx - \lambda \int_{\Omega} \frac{u T_k(u-T_h(v))}{|x|^{2s}} \, dx  \\
& \leq \int_{\Omega} \frac{\mu T_k(u-T_h(v))}{u^{\gamma}} \, dx + \int_{\Omega} f T_k(u-T_h(v)) \, dx,
\end{aligned}
\end{equation}
and
\begin{equation}
\label{ENTEq21}
\begin{aligned}
\int_{\{|v-T_h(u)|<k\}} & (-\Delta)^{\frac{s}{2}} v (-\Delta)^{\frac{s}{2}} (v-T_h(u)) \, dx  - \lambda \int_{\Omega} \frac{v T_k(v-T_h(u))}{|x|^{2s}} \, dx \\
& \leq \int_{\Omega} \frac{\mu T_k(v-T_h(u))}{v^{\gamma}} \, dx + \int_{\Omega} f T_k(v-T_h(u)) \, dx.
\end{aligned}
\end{equation}
Adding up the left-hand sides of \eqref{ENTEq20} and \eqref{ENTEq21} and restricting them to 
\begin{equation*} 
A_0^h=\{ x \in \Omega \,:\, |u-v|<k, \, |u|<h, \, |v|<h \},
\end{equation*}
we have the following estimate by using Hardy inequality
\begin{equation}
\label{ENTEq22}
\int_{A_0^h} |(-\Delta)^{\frac{s}{2}} (u-v)|^2 \, dx - \lambda \int_{A_0^h} \frac{(u-v)^2}{|x|^{2s}} \, dx \geq \frac{\Lambda_{N,s}-\lambda}{\Lambda_{N,s}} \int_{A_0^h} |(-\Delta)^{\frac{s}{2}} (u-v)|^2 \, dx.
\end{equation}
Also, summing the right-hand sides of \eqref{ENTEq20} and \eqref{ENTEq21} when restricted to $A^h_0$ gives 
\begin{equation}
\label{ENTEq23}
\int_{A_0^h} (u-v)(u^{-\gamma}-v^{-\gamma}) \mu \, dx \leq 0.
\end{equation}
Now, consider the set $ A_1^h=\{ x \in \Omega \,:\, |u-T_h(v)|<k, \, |v| \geq h \}$. 
When restricted to $A_1^h$, we have the following for the left-hand side of \eqref{ENTEq20}:
\begin{equation}
\label{ENTEq24}
\begin{aligned}
\int_{A_1^h} |(-\Delta)^{\frac{s}{2}} u|^2 \, dx - \lambda \int_{A_1^h} \frac{u(u-h)}{|x|^{2s}} \, dx & \geq \int_{A_1^h} |(-\Delta)^{\frac{s}{2}} u|^2 \, dx - \lambda \int_{A_1^h} \frac{u^2}{|x|^{2s}} \, dx \\
& \geq \frac{\Lambda_{N,s}-\lambda}{\Lambda_{N,s}} \int_{A_1^h} |(-\Delta)^{\frac{s}{2}} u|^2 \, dx \geq 0.
\end{aligned}
\end{equation}
On the other hand, when restricted to $A_1^h$, the right-hand side of \eqref{ENTEq20} is
\begin{equation}
\label{ENTEq24+++}
\int_{A_1^h} u^{-\gamma} (u-h) \mu \, dx + \int_{A_1^h} f (u-h) \, dx,
\end{equation}
which goes to zero as $h \to \infty$.

Finally on the remaining set $ A_2^h=\{ x \in \Omega \,:\, |u-T_h(v)|<k, \, |v| < h, \, |u| \geq h \}$, the left-hand side of \eqref{ENTEq20} is as follows
\begin{equation}
\label{ENTEq25}
\int_{A_2^h} (-\Delta)^{\frac{s}{2}} u (-\Delta)^{\frac{s}{2}} (u-v) \, dx - \lambda \int_{A_2^h} \frac{u (u-v)}{|x|^{2s}} \, dx,
\end{equation}
which goes to zero as $h \to \infty$.

The right-hand side of \eqref{ENTEq20}, when restricted to $A_2^h$, is as follows
\begin{equation}
\label{ENTEq25++}
\begin{aligned}
\int_{A_2^h} u^{-\gamma} (u-v) \mu \, dx + \int_{A_2^h} f(x) (u-v) \, dx,
\end{aligned}
\end{equation}
which also goes to zero as $h \to \infty$.

Similarly, we can estimate the left-hand side of \eqref{ENTEq21} on the sets $B_1^h=\{x \in \Omega \,:\, |v-T_h(u)|<k, \, |u| \geq h \}$ and $B_2^h=\{  x \in \Omega \,:\, |v-T_h(u)|<k, \, |u| < h, \, |v| \geq h  \}$ and find that
\begin{equation}
\label{ENTEq26}
\int_{B_1^h} |(-\Delta)^{\frac{s}{2}} v|^2 \, dx - \lambda \int_{B_1^h} \frac{v(v-h)}{|x|^{2s}} \, dx \geq \frac{\Lambda_{N,s}-\lambda}{\Lambda_{N,s}} \int_{B_1^h} |(-\Delta)^{\frac{s}{2}} v|^2 \, dx \geq 0,
\end{equation}
and
\begin{equation}
\label{ENTEq27}
\int_{B_2^h} (-\Delta)^{\frac{s}{2}} v (-\Delta)^{\frac{s}{2}} (v-u) \, dx - \lambda \int_{B_2^h} \frac{v (v-u)}{|x|^{2s}} \, dx \to 0, \qquad \text{as} \,\, h \to 0.
\end{equation}

On the other hand for the right-hand side of \eqref{ENTEq21} on the sets $B_1^h=\{x \in \Omega \,:\, |v-T_h(u)|<k, \, |u| \geq h \}$ and $B_2^h=\{  x \in \Omega \,:\, |v-T_h(u)|<k, \, |u| < h, \, |v| \geq h  \}$, we have: 
\begin{equation}
\label{ENTEq26++}
\int_{B_1^h} v^{-\gamma} (v-h) \mu \, dx +\int_{B_1^h} f (v-h) \, dx \to 0, \qquad \text{as} \,\, h \to 0,
\end{equation}
and
\begin{equation}
\label{ENTEq27++}
\int_{B_2^h} v^{-\gamma} (v-u) \mu \, dx + \int_{B_2^h} f (v-u) \, dx \to 0, \qquad \text{as} \,\, h \to 0.
\end{equation}
Putting all the estimates \eqref{ENTEq22}, \eqref{ENTEq23}, \eqref{ENTEq24}, \eqref{ENTEq24+++}, \eqref{ENTEq25}, \eqref{ENTEq25++}, \eqref{ENTEq26}, \eqref{ENTEq27}, \eqref{ENTEq26++}, and \eqref{ENTEq27++} together we obtain
\begin{equation*} 
\int_{A_0^h} |(-\Delta)^{\frac{s}{2}} (u-v)|^2 \, dx \leq \mathrm{o}(h), \qquad \text{as} \,\, h \to 0. 
\end{equation*}
Now, since $A_0^h$ goes to $\{ |u-v|<k \}$, as $h \to 0$ we have
\begin{equation*} 
\int_{\{|u-v|<k \}} |(-\Delta)^{\frac{s}{2}} (u-v)|^2 \, dx \leq 0, \qquad \forall k. 
\end{equation*}
Therefore $u \equiv v$, and the uniqueness is proved.

Now, we construct an entropy solution for the case $0<\gamma \leq 1$, $ \mu \in L^{\big(\frac{2_s^*}{1-\gamma}\big)'}(\Omega) \cap L^2(\Omega)$ and a datum of $f\in L^1(\Omega)$ such that satisfies the integrability condition \eqref{IntegrabilityC}. Let consider the following approximating problems:
\begin{equation}
\label{ENTEq30.}
\begin{cases}
(- \Delta)^s u_n = \lambda \dfrac{u_n}{|x|^{2s}}+ \dfrac{\mu_n}{(u_n+\frac{1}{n})^{\gamma}}+ f_n &   \mathrm{in} \,\, \Omega,\\ u_n>0 &   \mathrm{in} \,\, \Omega, \\ u_n=0 &   \mathrm{in} \,\,  \big(\mathbb{R}^N \setminus \Omega \big).
\end{cases}
\end{equation}
Here $\mu_n=T_n(\mu)$ and $f_n=T_n(f)$. The increasing behaviour of $ \mu_n(u_n+\frac{1}{n})^{-\gamma} + f_n$, and the monotonicity of the operator $(-\Delta)^s u - \lambda \frac{u}{|x|^{2s}}$ will ensure the existence of an increasing sequence of solutions to problems \eqref{ENTEq30.}. Testing \eqref{ENTEq30.} with $T_k(u_n-\phi)$ implies that $\{ T_k(u_n-\phi)\}_{n=1}^{\infty}$ is a bounded sequence in $X_0^s(\Omega)$ for each fixed $k$ and each fixed $\phi \in X_0^s(\Omega) \cap L^{\infty}(\Omega)$. Therefore, up to a subsequence $T_k(u_n-\phi) \to T_k(u-\phi)$ weakly in $X_0^s(\Omega)$ as $n \to \infty$, where $u$ is the weak solution to \eqref{Eq1} with $\mu \in L^{\big(\frac{2_s^*}{1-\gamma}\big)'}(\Omega) \cap L^2(\Omega)$. Also, since $\{T_k(u_n-\phi)\}_{n=1}^{\infty}$ is an increasing sequence of non-negative functions, once more the strict monotonicity of $(- \Delta)^s $ implies that $ T_k(u_n-\phi) \to T_k(u-\phi)$ strongly in $X_0^s(\Omega)$ (see for example \cite[Lemma 2.18]{MR3479207} for this compactness result). Now, using $T_k(u_n-\phi)$ as a test function in \eqref{ENTEq30.}, and noting that 
\begin{equation*} 
\begin{aligned}
 \int_{\Omega} \Bigg| \frac{\mu_nT_k(u_n-\phi)}{(u_n+\frac{1}{n})^{\gamma}} \Bigg| \, dx & \leq \int_{\Omega} \frac{\mu_n |T_k(u_n-\phi)|}{u_1^{\gamma}} \, dx \\
 & \leq \|\mu\|_{L^2(\Omega)} \int_{\Omega} \frac{T_k^2(u_n-\phi)}{\delta^{2s\gamma}} \, dx \\
 & \leq C_1 \big\| |T_k(u_n-\phi)| \big\|_{X_0^{s\gamma}(\Omega)} \\
& \leq C_2 \|T_k(u_n-\phi) \|_{X_0^s(\Omega)} \leq C_3<+ \infty, \qquad \text{uniformly in} \,\, n,
\end{aligned}
\end{equation*}
(because of $u_1 \sim c \delta^s$, near the boundary, and applying the H\"older and the fractional Hardy-Sobolev inequalities) we may pass to the limit and find an entropy solution even with the equalities instead of the inequalities in Definition \ref{NewDEFN}. Notice that from the above estimate and by Fatou's Lemma we deduce
$$ \int_{\Omega} \frac{T_k(u - \phi) \mu}{u^{\gamma}} \, dx < + \infty, $$
for any $\phi \in X_0^s(\Omega) \cap L^{\infty}(\Omega)$, and any $k>0$.

We end this section by a Calder\'on-Zygmund type property to solutions of problem \eqref{Eq1}. See \cite{MR2257147} for this property in the local case without the presence of singular nonlinearity and \cite{MR2592976} for the case without the Hardy potential. 

As mentioned before in Lemma \ref{Lem1}, any supersolution to \eqref{Eq1} is unbounded, i.e., $u(x) \gtrsim |x|^{-\beta}$ in a neighborhood of the origin. Now we have the following result, which says this rate is precisely the rate of the growth of $u$ for the regular data of $\mu$ and $f$.

\begin{theorem}
\label{theorem-calderon-zygmundP}
Let $0 \leq \mu,f \in L^m(\Omega)$, $m > \frac{N}{2s}$, and consider $u \in X_0^s(\Omega)$, as the weak energy solution to \eqref{Eq1} with $0 < \lambda < \Lambda_{N,s}$. Then $u(x) \leq C|x|^{-\beta}$ a.e. in $\Omega$.
\end{theorem}

\begin{proof}
We follow \cite[Theorem 4.1]{MR3479207}. Also see \cite[Lemma 3.3]{MR2592976}. Let $k \geq 1$. By the change of variable $v(x):=|x|^{\beta} u(x)$, it can be checked that $v$ solves:
\begin{equation}
\label{Eq1GRSV}
\begin{cases}
L_{\beta} v = \dfrac{\mu}{|x|^{\beta(1-\gamma)} v^{\gamma}} + |x|^{-\beta} f & \quad \mathrm{in} \,\, \Omega,\\ v>0 & \quad \mathrm{in} \,\, \Omega, \\ v=0 & \quad \mathrm{in} \,\, \big(\mathbb{R}^N \setminus \Omega \big),
\end{cases}
\end{equation}
where the operator $L_{\beta}$ is as follows:
\begin{equation*} 
L_{\beta} v := C_{N,s} \mathrm{P.V.} \int_{\mathbb{R}^N} \frac{v(x)-v(y)}{|x-y|^{N+2s}} \frac{dy}{|x|^{\beta} |y|^{\beta}}. \end{equation*}
See \cite[Section 2]{MR3479207} for the properties of this operator and the associated weighted fractional Sobolev space.

Using $G_k(v)$ as a test function in \eqref{Eq1GRSV}, and following the proof of \cite[Theorem 4.1]{MR3479207} we obtain
\begin{equation}
\label{Eq15.2-ZYG}
\begin{aligned}
\frac{C_{N,s}}{2} \iint_{D_{\Omega}} \frac{|G_k(v(x))-G_k(v(y))|^2}{|x-y|^{N+2s}} \, \frac{dx}{|x|^{\beta}} \frac{dy}{|x|^{\beta}} & \leq  \int_{A_k} |x|^{\beta \gamma} \frac{\mu}{v^{\gamma}} G_k(v) \, \frac{dx}{|x|^{\beta}} \\
& \quad + \int_{A_k} f G_k(v) \, \frac{dx}{|x|^{\beta}},
\end{aligned}
\end{equation}
where $A_k:=\{ x \in \Omega \,:\, v(x) \geq k\}$. Applying the weighted Sobolev inequality \cite[Proposition 2.11]{MR3479207} in the left-hand side of \eqref{Eq15.2-ZYG}, and noting that $|x|^{\beta \gamma} \leq C_2$, in $\Omega$, gives
\begin{equation*} 
C_1 \|G_k(v)\|_{L^{2_s^*}(\Omega, |x|^{-\beta} \, dx)}^2 \leq C_2 \int_{A_k} \frac{\mu}{v^{\gamma}} G_k(v) \, \frac{dx}{|x|^{\beta}} + \int_{A_k} f G_k(v) \, \frac{dx}{|x|^{\beta}}.
\end{equation*}
For the first term in the right-hand side of the above inequality, by using the H\"older inequality we get
\begin{equation*} 
\Bigg| \int_{A_k} \frac{\mu}{v^{\gamma}} G_k(v) \, \frac{dx}{|x|^{\beta}} \Bigg| \leq k^{-\gamma} \|\mu\|_{L^m(\Omega)} \|G_k(v)\|_{L^{2_s^*}(\Omega, |x|^{-\beta} \, dx)} |A_k|^{1-\frac{1}{2_s^*}-\frac{1}{m}}.
\end{equation*}
Similarly, for the second term
\begin{equation*} 
\Bigg| \int_{A_k} f G_k(v) \, \frac{dx}{|x|^{\beta}} \Bigg| \leq \|f\|_{L^m(\Omega)} \|G_k(v)\|_{L^{2_s^*}(\Omega, |x|^{-\beta} \, dx)} |A_k|^{1-\frac{1}{2_s^*}-\frac{1}{m}}.
\end{equation*}
Putting the results together, we obtain
\begin{equation}
\label{Eq15.2-ZYG10}
\|G_k(v)\|_{L^{2_s^*}(\Omega, |x|^{-\beta} \, dx)} \leq C_3 |A_k|^{1-\frac{1}{2_s^*}-\frac{1}{m}}.
\end{equation}
On the other hand, since $\Omega$ is bounded, there exists a constant $C_4 > 0$ such that
\begin{equation}
\label{Eq15.2-ZYG11}
\|G_k(v)\|_{L^{2_s^*}(\Omega, |x|^{-\beta} \, dx)} \geq C_4 \|G_k(v)\|_{L^{2_s^*}(\Omega)}.
\end{equation}
Moreover, for any $z>k$, we have that $A_z \subset A_k$ and $G_k(v) \chi_{A_z} \geq (z - k)$. Thus from \eqref{Eq15.2-ZYG10} and \eqref{Eq15.2-ZYG11} we have
\begin{equation*} 
(z-k) |A_z|^{\frac{1}{2_s^*}} \leq C_5 |A_k|^{1-\frac{1}{2_s^*}-\frac{1}{m}},
\end{equation*}
or equivalently
\begin{equation*} 
|A_z| \leq C_6 \frac{|A_k|^{2_s^*(1-\frac{1}{2_s^*}-\frac{1}{m})}}{(z-k)^{2_s^*}}.
\end{equation*}
Now by invoking \cite[Lemma 2.23]{MR3479207} with the choice of $\psi(h):=|A_h|$, and noting that $ 2_s^*(1-\frac{1}{2_s^*}-\frac{1}{m}) >1$, because of $m> \frac{N}{2s}$, we obtain that there exists $k_0$ such that $\psi(k) \equiv 0$, for any $k \geq k_0$. Thus $v(x) \leq k_0$, a.e. in $\Omega$. This means that $u(x) \leq k_0 |x|^{-\beta}$, a.e. in $\Omega$.
\end{proof}

\section{The parabolic case and a stabilization result}
\label{Section4}
In this section, we study on the following evolution problem 
\begin{equation}
\label{EqP.1}
\begin{cases}
u_t+(-\Delta )^s u = \lambda \dfrac{u}{|x|^{2s}} +\dfrac{1}{u^{\gamma}}+f(x,t) & \quad \mathrm{in} \,\, \Omega \times (0,T), \\ u>0 & \quad \mathrm{in} \,\, \Omega \times (0,T), \\ u =0 & \quad \mathrm{in} \,\, (\mathbb{R}^N \setminus \Omega) \times (0,T), \\ u(x,0)=u_0 & \quad \mathrm{in} \,\, \mathbb{R}^N,
\end{cases}
\end{equation}
where $u_0 \in X_0^s(\Omega)$ satisfies an appropriate cone condition which will be precised later. In what follows, we will mention an existence and uniqueness and also a stabilization result to problem \eqref{EqP.1}. 

First of all, we define a notion of a weak solution. Before it, we need the following class of test functions. 
\begin{equation*} 
\mathcal{A}(\Omega_T) := \Big\{u \,\,:\,\, u \in L^2(\Omega \times (0,T)), \,\, u_t \in L^2(\Omega \times (0,T)), \,\, u \in L^{\infty} (0,T; X_0^s(\Omega)) \Big\}.
\end{equation*}

Notice that Aubin-Lions-Simon Lemma, see \cite{MR916688}, implies that the following embedding is compact.
\begin{equation}
\label{EqP.105}
\mathcal{A}(\Omega_T) \hookrightarrow C([0,T]; L^2(\Omega)).
\end{equation}

\begin{definition}
Assume $u_0 \in L^2(\Omega)$, and $f \in L^{2}(\Omega \times (0,T))$. We say that $u \in \mathcal{A}(\Omega_T)$ is a weak supersolution (subsolution) to problem \eqref{EqP.1} if 
\begin{itemize} 
\item
for every $K \Subset \Omega \times (0,T)$, there exists $C_K>0$ such that $u(x,t) \geq C_K$ a.e. in $K$ and also $u \equiv 0$ in $\big(\mathbb{R}^N \setminus \Omega \big) \times [0,T)$;
\item
for every non-negative $\phi \in \mathcal{A}(\Omega_T)$, we have
\begin{equation*} 
\begin{aligned}
\int_{0}^{T} \int_{\Omega} u_t \phi \, dx dt &+ \int_{0}^{T} \int_{\mathbb{R}^N} (-\Delta)^{\frac{s}{2}} u (-\Delta)^{\frac{s}{2}} \phi \, dx dt \\
& \geq (\leq ) \lambda \int_{0}^{T} \int_{\Omega} \dfrac{u\phi}{|x|^{2s}}  \, dx dt + \int_{0}^{T} \int_{\Omega} \dfrac{ \phi}{u^{\gamma}} \, dx dt + \int_{0}^{T} \int_{\Omega} f \phi \, dx dt;
\end{aligned}
\end{equation*}
 and also together with this extra assumption that the second term on the right-hand side of the above inequality be finite for any $\phi \in \mathcal{A}(\Omega_T)$. The well-posedness of the second term on the right-hand side will be clear after the construction of solution.
\item
$u(x,0) \geq (\leq) u_0(x)$ a.e. in $\Omega$. 
\end{itemize}
If $u$ is a weak supersolution and subsolution then we say that $u$ is a weak solution. Notice that by the embedding \eqref{EqP.105}, the initial condition $u(x,0)=u_0$ make sense.
\end{definition}

Before outlining our theorems, we need to define the following sets:
\begin{itemize}
\item
Let $\mathcal{U}^{\mathrm{Sing}}_{\gamma}$ be the set of all functions in $L^{2}(\Omega)$ such that there exists $k_1>0$ such that
\begin{equation*} 
\begin{cases}
k_1 \delta^{s} (x) \leq |x|^{\beta} u(x), \qquad & 0<\gamma<1, \\
k_1 \delta^{s} (x) \left( \ln \left(\dfrac{r}{\delta^s (x)} \right) \right)^{\frac{1}{2}} \leq |x|^{\beta} u(x), \qquad & \gamma=1, \\
k_1 \delta^{\frac{2s}{\gamma+1}} (x) \leq |x|^{\beta} u(x), \qquad & \gamma>1,
\end{cases}
\end{equation*}
where $r> \mathrm{diam}(\Omega)$.
\item
Let $W(\Omega):= \{\phi \in C(\overline{\Omega} \setminus \{0\}) \,\, : \,\, |x|^{\beta} \phi \in C(\overline{\Omega}) \}$, which is equipped with the $L^{\infty}(\Omega, |x|^{\beta} \, dx)$ norm, i.e.
\begin{equation*} 
\|u\|_{L^{\infty}(\Omega, |x|^{\beta} \, dx)} := \mathrm{ess\,sup\,} \Big\{ |x|^{\beta} |u(x)| \,\, : \,\, x \in \Omega \Big\}.
\end{equation*}

\end{itemize}

Also, we need to the following definition.
\begin{definition}
We say that $u(t) \in \mathcal{U}^{\mathrm{Sing}}_{\gamma}$ uniformly for each $t \in [0,T]$ when there exists $\psi_1, \psi_2 \in \mathcal{U}^{\mathrm{Sing}}_{\gamma}$ such that $\psi_1(x) \leq u(x,t) \leq \psi_2(x)$ a.e. $(x,t) \in \Omega \times [0,T]$.
\end{definition}

\begin{theorem}
\label{Thm3}
Let $0 \leq g \in L^{\infty}(\Omega, |x|^{\beta} \, dx)$, $0<\lambda < \Lambda_{N,s}$, and $\theta>0$. Then the following problem has a unique weak energy solution $u_{\theta} \in X_0^s(\Omega) \cap \mathcal{U}^{\mathrm{Sing}}_{\gamma}$ for any $0<\gamma  \leq 1$, or $\gamma>1$ with $2s(\gamma-1)< \gamma+1$.
\begin{equation}
\label{EqP.2}
\begin{cases}
u+\theta \Big((-\Delta )^s u - \lambda \dfrac{u}{|x|^{2s}} - \dfrac{1}{u^{\gamma}} \Big) = g & \quad \mathrm{in} \,\, \Omega,\\ u>0 & \quad \mathrm{in} \,\, \Omega, \\ u=0 & \quad \mathrm{in} \,\, \big(\mathbb{R}^N \setminus \Omega \big).
\end{cases}
\end{equation}
Moreover, there exists a positive constant $ \lambda_*<\Lambda_{N,s} $ such that for any $\lambda \in (0,\lambda_*)$, this unique solution also belongs to $W(\Omega)$.
\end{theorem}

\begin{proof}
We follow the proof of \cite[Theorem 2.4]{MR3842325}. For any $\epsilon>0$, let consider the following approximating problem:
\begin{equation}
\label{EqP.2HHH1}
\begin{cases}
u_{\epsilon,\theta}+\theta \Big((-\Delta )^s u_{\epsilon,\theta} - \lambda \dfrac{u_{\epsilon,\theta}}{|x|^{2s}} - \dfrac{1}{(u_{\epsilon,\theta}+\epsilon)^{\gamma}} \Big) = g & \quad \mathrm{in} \,\, \Omega,\\ u_{\epsilon,\theta}>0 & \quad \mathrm{in} \,\, \Omega, \\ u_{\epsilon,\theta}=0 & \quad \mathrm{in} \,\, \big(\mathbb{R}^N \setminus \Omega \big).
\end{cases}
\end{equation}
The existence of a unique energy solution easily follows by the classical variational methods. Indeed let $X_0^s(\Omega)^+:=\{u \in X_0^s(\Omega) \,|\, u \geq 0 \}$, and consider the corresponding energy functional to problem \eqref{EqP.2HHH1} as follows:
\begin{equation*}  
\begin{aligned}
I_{\epsilon,\theta}(u)  =  \frac{1}{2} \int_{\Omega} u^2 \, dx & + \frac{\theta C_{N,s}}{4} \|u\|_{X_0^s(\Omega)}^2 -\frac{\theta \lambda}{2} \int_{\Omega} \frac{u^2}{|x|^{2s}} \, dx \\ 
& - \frac{\theta}{1-\gamma} \int_{\Omega} (u+\epsilon)^{1-\gamma} \, dx - \int_{\Omega} gu \, dx, \qquad u \in X_0^s(\Omega)^+.
\end{aligned}
\end{equation*}
Notice that the last term is well-defined since $g \in L^{\infty}(\Omega, |x|^{\beta} \, dx) \subset L^2(\Omega)$. Using Hardy inequality, one can show that this functional $I_{\epsilon,\theta}: X_0^s(\Omega)^+ \to \mathbb{R}$ is weakly lower semi-continuous, coercive and strictly convex. Since $X_0^s(\Omega)^+$ is a closed subspace of the reflexive space $X_0^s(\Omega)^+$, therefore the existence of a unique minimizer is obvious by the classical theory (for instance see \cite[Chapter 1]{MR2722059}). Therefore, as a consequence, we get the existence of a unique energy solution to problem \eqref{EqP.2HHH1}.  

Let $0< \epsilon_1 \leq \epsilon_2$. We want to show that $u_{\epsilon_2, \theta} \leq u_{\epsilon_1, \theta}$ a.e. in $\Omega$. This easily follows by subtracting the weak formulations of $u_{\epsilon_i, \theta}$, $i=1,2$, and using $(u_{\epsilon_2, \theta} - u_{\epsilon_1, \theta})^+$ as a test function which together with the Hardy inequality implies $(u_{\epsilon_2, \theta} - u_{\epsilon_1, \theta})^+ \equiv 0$, a.e. in $\Omega$. Now let $w \in X_0^s(\Omega) \cap \mathcal{U}^{\mathrm{Sing}}_{\gamma}$ be the unique energy solution to
\begin{equation*} 
\begin{cases}
(-\Delta )^s w = \lambda \dfrac{w}{|x|^{2s}} +w^{-\gamma} & \quad \mathrm{in} \,\, \Omega,\\ w>0 & \quad \mathrm{in} \,\, \Omega, \\ w=0 & \quad \mathrm{in} \,\, \big(\mathbb{R}^N \setminus \Omega \big).
\end{cases}
\end{equation*}
Notice that for the general $\gamma>1$, we only know that $w \in X_{\mathrm{loc}}^s(\Omega)$. But since $2s(\gamma-1)< \gamma+1$, thanks to Remark \ref{Reemark}, we get $w \in X_0^s(\Omega)$ too. 

Now define $\overline{u}:=Mw$, for some $M>1$. Because of the same singular behaviour of $w$ and $g$ near the origin, and noting that $g$ is bounded, near the boundary, $\partial \Omega$, and $w$ behaves as $c \delta^s$, near the boundary, we can choose $M$ large enough (independent of $\epsilon$) such that
\begin{equation*} 
\begin{aligned}
\overline{u}+ \theta \Bigg( (-\Delta)^s \overline{u} - \lambda \frac{\overline{u}}{|x|^{2s}} - \frac{1}{(\overline{u}+\epsilon)^{\gamma}} \Bigg) & = Mw+ \theta \Bigg( \frac{M}{w^{\gamma}} - \frac{1}{(Mw+\epsilon)^{\gamma}} \Bigg) \\
& \geq  Mw+ \theta \Bigg( \frac{1}{(Mw)^{\gamma}} - \frac{1}{(Mw+\epsilon)^{\gamma}} \Bigg)  \\
& >g, \qquad \text{in} \,\, \Omega.
\end{aligned}
\end{equation*}
Since $A_{\theta} : X_0^s(\Omega) \cap \mathcal{U}^{\mathrm{Sing}}_{\gamma} \to X^{-s}(\Omega)$, $A_{\theta}(u):=u+\theta \big( (-\Delta)^s u - \lambda \frac{u}{|x|^{2s}} - u^{-\gamma} \big) $, is a strictly monotone operator for $0< \lambda < \Lambda_{N,s}$ (this strict monotonicity is the easy consequence of \cite[Lemma 3.1]{MR3842325} and the Hardy inequality) therefore $ u_{\epsilon, \theta} \leq \overline{u}$.
Thus $u_{\theta} \leq \overline{u}$, where $ u_{\theta}:= \lim_{\epsilon \to 0^+} u_{\epsilon, \theta}$. This implies that $u_{\theta}$ is a very weak (distributional) solution to problem \eqref{EqP.2}, i.e. 
\begin{equation}
\label{EqP.2HHH2}
\int_{\Omega} u_{\theta} \phi \, dx +\theta \Bigg( \int_{\mathbb{R}^N} u_{\theta} (-\Delta )^s \phi \, dx - \lambda \int_{\Omega} \dfrac{u_{\theta}}{|x|^{2s}} \phi \, dx - \int_{\Omega} \dfrac{\phi}{u_{\theta}^{\gamma}} \, dx \Bigg) = \int_{\Omega} g \phi \, dx,
\end{equation}
for any $\phi \in \mathcal{T}(\Omega)$. But in fact, we want to show that $u_{\theta}$ is an energy solution. For this purpose let $\underline{u}:=mw$, for some $m>0$. If we choose $m$ small enough such that
\begin{equation*} 
m^{\gamma+1} \Bigg( 1+ \frac{w^{\gamma+1}}{\theta} \Bigg) \leq 1+ m^{\gamma} \frac{g w^{\gamma}}{\theta}, \qquad \text{in} \,\, \Omega,
\end{equation*}
(which is possible by taking into consideration the behavior of $w$ and $g$ near the origin and the boundary,  $\partial\Omega$) then $\underline{u}$ will be a subsolution to problem \eqref{EqP.2} and with the similar arguments as in above we obtain $\underline{u} \leq u_{\theta}$ a.e. in $\Omega$. Thus $\underline{u} \leq u_{\theta} \leq \overline{u}$, which implies that $u_{\theta} \in \mathcal{U}^{\mathrm{Sing}}_{\gamma}$. On the other hand, by invoking the Hardy inequality and also because of the restrictions $0<\gamma \leq 1$, or $\gamma >1$ with $2s(\gamma-1)< \gamma+1$, a density argument shows that \eqref{EqP.2HHH2} holds for all $\phi \in X_0^s(\Omega)$ (see Remark \ref{Reemark}). This means that $u_{\theta} \in X_0^s(\Omega)$ is the unique energy solution to problem  \eqref{EqP.2}. 

Now, let $g \in L^m(\Omega)$, $m > \frac{N}{2s}$, which is possible if $m \beta < N$, or equivalently $\alpha > \frac{N-2s}{2}-\frac{N}{m}$. Since $\lambda=\lambda(\alpha)$, given by \eqref{Eq3...01}, is a continuous decreasing function for $\alpha \in [0, \frac{N-2s}{2})$, this recent condition is equivalent to $0< \lambda < \lambda_*$, for some $\lambda_*<\Lambda_{N,s}$. Thus Comparison Principle for the fractional Laplacian operator together with Theorem \ref{theorem-calderon-zygmundP} gives $u(x) \leq C |x|^{-\beta}$ a.e. in $\mathbb{R}^N$. Now, the interior regularity theory for the fractional Laplacian, that follows from \cite[Proposition 1.1]{MR3168912}, implies that $u \in C(\tilde{\Omega} \setminus B_{\epsilon}(0))$, for any $\tilde{\Omega} \Subset \Omega$ and any $\epsilon>0$ small enough. Moreover, by following the proof of \cite[Theorem 1.4]{MR3797614} we obtain the continuity of $u$ up to the boundary of $\Omega$. This completes the proof.

\end{proof}

Thanks to Hardy inequality and following the idea of \cite[Theorem 4.1]{MR3842325}, i.e. applying the semi-discretization in time with implicit Euler method, and also invoking the result of Theorem \ref{Thm3}, we will obtain the following existence result to problem \eqref{EqP.1}.  

\begin{theorem} 
\label{Thm4}
Let $ s \in (0,1)$, $0<\gamma \leq 1$, or $\gamma>1$ with $2s(\gamma-1)< \gamma+1$, and $0 < \lambda <\Lambda_{N,s}$. Also assume that $u_0 \in X_0^s(\Omega) \cap \mathcal{U}^{\mathrm{Sing}}_{\gamma}$, and $0 \leq f(x, t) \leq |x|^{\gamma\beta}$, $0 \leq t \leq T$. Then there is a unique positive weak solution in $\mathcal{A}(\Omega_T) \cap \mathcal{U}^{\mathrm{Sing}}_{\gamma}$ to problem \eqref{EqP.1}. Moreover, $u$ belongs to $C([0,T], X_0^s(\Omega))$, and $u(t) \in \mathcal{U}^{\mathrm{Sing}}_{\gamma}$ uniformly for each $t \in [0,T]$, and also for any $t \in [0,T]$
\begin{equation}
\label{Eq.P1ESTIMATEP}
\begin{aligned}
& \int_{0}^{t} \int_{\Omega} \Big|\frac{\partial u}{\partial \tau} \Big|^2 \, dx d\tau + \frac{C_{N,s}}{2} \|u(x,t)\|^2_{X_0^s(\Omega)} - \lambda \int_{\Omega} \dfrac{u^2(x,t)}{|x|^{2s}} \, dx \\
& \qquad \qquad - \frac{1}{1-\gamma} \int_{\Omega} u^{1-\gamma}(x,t) \, dx \\
& = \int_{0}^{t} \int_{\Omega} f(x,t) \frac{\partial u}{\partial \tau} \, dx d\tau + \frac{C_{N,s}}{2} \|u_0(x)\|^2_{X_0^s(\Omega)} - \lambda \int_{\Omega} \dfrac{u^2_0}{|x|^{2s}} \, dx \\
& \qquad \qquad - \frac{1}{1-\gamma} \int_{\Omega} u_0^{1-\gamma}(x) \, dx.
\end{aligned}
\end{equation}
In addition, if $0 < \lambda < \lambda_*$ ($\lambda_*$ is as in Theorem \ref{Thm3}), and $u_0 \in \overline{\mathcal{D}(L)}^{L^{\infty}(\Omega, |x|^{\beta} \, dx)}$, where
$$\mathcal{D}(L):=\Big\{ v \in X_0^s(\Omega) \cap \mathcal{U}^{\mathrm{Sing}}_{\gamma} \cap W(\Omega) \Big| L(v):=(-\Delta)^s v -\lambda \frac{v}{|x|^{2s}}- \frac{1}{v^{\gamma}} \in L^{\infty}(\Omega, |x|^{\beta} \, dx) \Big\},$$ 
then the solution obtained above belongs to $C([0,T]; W(\Omega))$.
\end{theorem}

\begin{remark}
By invoking \cite[Proposition 5.3]{MR3492734}, it is straightforward to obtain that if $\lambda>\Lambda_{N,s}$, then problem \eqref{EqP.1} does not have any solution. Moreover, the similar complete blow-up phenomenon occurs as in the stationary case.
\end{remark}

Finally, the following theorem is about a stabilization result to problem \eqref{EqP.1}. By stabilization, we mean that if $\hat{u}(x)$ is the unique solution to the stationary problem with the datum of $f(x)$, then $u(x,t)$, the solution to the parabolic problem, converges to $\hat{u}(x)$, as $t \to \infty$.

\begin{theorem} 
\label{Thm5}
Let $ s \in (0,1)$, $0<\gamma \leq 1$, or $\gamma>1$ with $2s(\gamma-1)< \gamma+1$, and $0 < \lambda <\lambda_*$. Also assume that $u_0 \in \overline{\mathcal{D}(L)}^{L^{\infty}(\Omega, |x|^{\beta} \, dx)}$, and $0 \leq f(x,t)=f(x) \leq |x|^{\gamma\beta}$, $0 \leq t \leq T$. Then if $u(x,t)$ is the unique positive weak solution to problem \eqref{EqP.1}, then 
$$ u(x,t) \to \hat{u}(x), \qquad \text{in} \,\, L^{\infty}(\Omega, |x|^{\beta} \, dx) \,\, \text{as} \,\,  t \to + \infty,$$
where $\hat{u}$ is the unique weak solution to \eqref{Eq1} with $\mu \equiv 1$.
\end{theorem}

Since proofs of the theorems in this section are essentially the same as proofs of the corresponding ones in \cite{MR3842325}, we will give them in the appendix.

\section{Appendix}
Here we give the proofs of Theorem \ref{Thm4} and Theorem \ref{Thm5}.
\begin{proof}[Proof of Theorem \ref{Thm4}]
We will follow the proofs of \cite[Theorem 4.1, Theorem 4.2 and Proposition 2.8]{MR3842325}. 

Let $\eta_t=\frac{T}{n}$ and for $0\leq k \leq n$, define $t_k=k\eta_t$ and 
$$ f_k(x):=\frac{1}{\eta_t} \int_{t_{k-1}}^{t_k} f(x,\tau) \, d\tau, \qquad \forall x \in \Omega. $$
Also, define
\begin{equation*} 
f_{\eta_t}(x,t):=
\begin{cases}
f_1(x) \qquad & 0 \leq t < t_1, \\
f_2(x) \qquad & t_1 \leq t < t_2, \\
\,\,\,\, \vdots \qquad & \qquad \vdots \\
f_n(x) \qquad & t_{n-1} \leq t < t_n.
\end{cases}
\end{equation*}
Clearly we have $f_{\eta_t}(\cdot,t) \in L^{\infty}(\Omega, |x|^{\beta} \, dx) \subset L^2(\Omega)$, $t \in [0,T]$, and for $ 1 < p < +\infty$, 
\begin{equation}
\label{EqP.3} 
\| f_{\eta_t} \|_{L^p(\Omega \times (0,T))} \leq \left( |\Omega| T \right)^{\frac{1}{p}} \| f \|_{L^p(\Omega \times (0,T))},
\end{equation}
Now, let $\theta=\eta_t$, and $g =\eta_t f_{k} + u_{k-1} \in L^{\infty}(\Omega, |x|^{\beta} \, dx)$ in problem \eqref{EqP.2}. Then, Theorem \ref{Thm3} implies the existence of $u_k \in X_0^s(\Omega) \cap \mathcal{U}^{\mathrm{Sing}}_{\gamma}$ as a solution to the following problem:
\begin{equation}
\label{EqP.4}
\begin{cases}
\dfrac{u_k-u_{k-1}}{\eta_t}+\Big((-\Delta )^s u_k - \lambda \dfrac{u_k}{|x|^{2s}} - \dfrac{1}{u_k^{\gamma}} \Big) = f_k & \quad \mathrm{in} \,\, \Omega,\\ u_k>0 & \quad \mathrm{in} \,\, \Omega, \\ u_k=0 & \quad \mathrm{in} \,\, \big(\mathbb{R}^N \setminus \Omega \big),
\end{cases}
\end{equation}
where the above iteration starts from the initial condition of problem \eqref{EqP.1}, i.e. $u_0(x)$.

Now, for $1 \leq k \leq n$, and $ t \in [t_{k-1},t_k)$, inspired by the implicit Euler method, we define
\begin{equation*} 
\begin{cases}
u_{\eta_t}(x,t):=u_k(x), \\
\tilde{u}_{\eta_t}(x,t):=\dfrac{u_k(x)-u_{k-1}(x)}{\eta_t}(t-t_{k-1})+u_{k-1}(x).
\end{cases}
\end{equation*}
The funtions $u_{\eta_t}$ and $\tilde{u}_{\eta_t}$ satisfies 
\begin{equation}
\label{EqP.6}
\frac{\partial \tilde{u}_{\eta_t}}{\partial t} + \Big((-\Delta )^s u_{\eta_t} - \lambda \dfrac{u_{\eta_t}}{|x|^{2s}} - \dfrac{1}{u_{\eta_t}^{\gamma}} \Big) = f_{\eta_t}.
\end{equation}
Now, in what follows, we establish some uniform estimates in $\eta_t$ for $u_{\eta_t}$ and $\tilde{u}_{\eta_t}$. 

Multiplying \eqref{EqP.4} by $\eta_t u_k$, integrating over $\mathbb{R}^N$ and summing from $k=1$ to $n' \leq n$, using Young's inequality,  \eqref{EqP.3} and the embedding \eqref{Eq3.350} we get for a constant $C>0$
\begin{equation}
\label{EqP.7}
\begin{aligned}
\sum_{k=1}^{n'} \int_{\Omega} (u_k-u_{k-1}) u_k \, dx & + \eta_t \sum_{k=1}^{n'} \Bigg( \frac{C_{N,s}}{2} \|u_k\|^2_{X_0^s(\Omega)}- \lambda \int_{\Omega} \frac{(u_k)^2}{|x|^{2s}} \, dx \\
& \qquad \qquad - \int_{\Omega} \frac{1}{u_k^{\gamma-1}} \, dx \Bigg) \\
&= \eta_t \sum_{k=1}^{n'} \int_{\Omega} f_ku_k \, dx \\
& \leq \eta_t \sum_{k=1}^{n'} \int_{\Omega} \frac{|f_k|^2}{2} \, dx + \eta_t \sum_{k=1}^{n'} \int_{\Omega} \frac{|u_k|^2}{2} \, dx \\
& \leq \frac{T|\Omega|}{2} \|f\|^2_{L^{\infty}(\Omega \times (0,T))} + \frac{C \eta_t}{2} \sum_{k=1}^{n'} \|u_k\|^2_{X_0^s(\Omega)}.
\end{aligned}
\end{equation}
For the first term in the left-hand side of \eqref{EqP.7}, similar to (2.7) in the proof of \cite[Theorem 0.9]{MR2891356}, we have the following equality
\begin{equation}
\label{EqP.8}
\begin{aligned}
\sum_{k=1}^{n'} \int_{\Omega} (u_k-u_{k-1}) u_k \, dx & = \frac{1}{2} \sum_{k=1}^{n'} \int_{\Omega} |u_k-u_{k-1}|^2 \, dx \\
& \quad + \frac{1}{2} \int_{\Omega} |u_{n'}|^2 \, dx - \frac{1}{2} \int_{\Omega} |u_0|^2 \, dx.
\end{aligned}
\end{equation}
Now, let $w \in X_0^s(\Omega) \cap \mathcal{U}^{\mathrm{Sing}}_{\gamma}$ solves 
\begin{equation*} 
\begin{cases}
(-\Delta )^s w = \lambda \dfrac{w}{|x|^{2s}}+\dfrac{1}{w^{\gamma}} & \quad \mathrm{in} \,\, \Omega,\\ w>0 & \quad \mathrm{in} \,\, \Omega, \\ w=0 & \quad \mathrm{in} \,\, \big(\mathbb{R}^N \setminus \Omega \big),
\end{cases}
\end{equation*}
and define $\underline{u}=mw$, $m>0$, and $\overline{u}=Mw$, $M>0$. By a direct computation we have
\begin{equation*} 
(-\Delta )^s \underline{u} - \lambda \dfrac{\underline{u}}{|x|^{2s}}-\dfrac{1}{\underline{u}^{\gamma}} = \frac{m^{\gamma+1}-1}{m^{\gamma} w^{\gamma}},
\end{equation*}
and
\begin{equation*} 
(-\Delta )^s \overline{u} - \lambda \dfrac{\overline{u}}{|x|^{2s}}-\dfrac{1}{\overline{u}^{\gamma}} = \frac{M^{\gamma+1}-1}{M^{\gamma} w^{\gamma}}. 
\end{equation*}
Since $w$ behaves as $c_1|x|^{-\beta}$ near the origin and behaves as $c_2\delta^s$, near the boundary, $\partial \Omega$, we can choose $m>0$ small enough, and $M>0$ large enough, such that
\begin{equation*} 
\begin{cases}
(-\Delta)^s \underline{u} - \lambda \dfrac{\underline{u}}{|x|^{2s}} - \dfrac{1}{\underline{u}^{\gamma}} \leq - |x|^{\gamma \beta} & \quad \mathrm{in} \,\, \Omega, \\
\underline{u}=0 & \quad \mathrm{in} \,\, \big(\mathbb{R}^N \setminus \Omega \big), 
\end{cases}
\end{equation*}
and
\begin{equation*} 
\begin{cases}
(-\Delta)^s \overline{u} - \lambda \dfrac{\overline{u}}{|x|^{2s}} - \dfrac{1}{\overline{u}^{\gamma}}  \geq |x|^{\gamma \beta}  & \quad  \mathrm{in} \,\, \Omega, \\
\overline{u}=0 & \quad \mathrm{in} \,\, \big(\mathbb{R}^N \setminus \Omega \big). 
\end{cases} 
\end{equation*}
Since $ u_0 \in \mathcal{U}^{\mathrm{Sing}}_{\gamma}$, we can choose $\underline{u}$ and $\overline{u}$ such that it satisfies the above inequalities and $\underline{u} \leq u_0 \leq \overline{u}$. From the monotonicity of the operator $ (-\Delta)^s u - \lambda \frac{u}{|x|^{2s}} - u^{-\gamma}$, and applying it iteratively we get $\underline{u} \leq u_k \leq \overline{u}$, for all $k$. This implies for a.e. $(x,t) \in [0,T] \times \Omega$,
\begin{equation}
\label{EqP.10}
\underline{u}(x) \leq u_{\eta_t}, \qquad \tilde{u}_{\eta_t}(x,t) \leq \overline{u}(x).
\end{equation}
Thus $u_{\eta_t}, \tilde{u}_{\eta_t} \in \mathcal{U}^{\mathrm{Sing}}_{\gamma}$ uniformly for each $t \in [0,T]$. Now, for the singular term in \eqref{EqP.7}, we can estimate as follows:
\begin{equation}
\label{EqP.10.1}
\eta_t \sum_{n=1}^{n'} \int_{\Omega} \frac{1}{u_k^{\gamma}} \, dx \leq
\begin{cases}
T \int_{\Omega} \overline{u}^{1-\gamma} \, dx < +\infty, \quad & 0<\gamma  \leq 1, \\
T \int_{\Omega} \underline{u}^{1-\gamma} \, dx < +\infty, \quad & \gamma>1, \,\, \mathrm{with} \,\, 2s(\gamma-1)<\gamma+1.
\end{cases}
\end{equation}
By the definition of $u_{\eta_t}$ and $\tilde{u}_{\eta_t}$, and noting that $u_k \in L^{\infty}(\Omega, |x|^{\beta} \, dx)$, for all $k$, we obtain
\begin{equation}
\label{EqP.11}
u_{\eta_t}, \tilde{u}_{\eta_t} \,\, \text{are bounded in} \,\, L^{\infty}([0,T];L^{\infty}(\Omega, |x|^{\beta} \, dx)).
\end{equation} 
On the other hand, for $t \in [t_{k-1},t_k)$, we have
\begin{equation*} 
\begin{aligned}
\| \tilde{u}_{\eta_t}(t,\cdot) \|_{X_0^s(\Omega)} & = \Big\| \frac{(t-t_{k-1})}{\eta_t} u_k + \frac{\eta_t-t+t_{k-1}}{\eta_t}u_{k-1} \Big\|_{X_0^s(\Omega)} \\
& \leq \| u_k \|_{X_0^s(\Omega)}+ \| u_{k-1} \|_{X_0^s(\Omega)}.
\end{aligned}
\end{equation*}
Integrating both sides of \eqref{EqP.7} over $(t_{k-1},t_k)$ and using the above estimates, the Hardy Inequality and \eqref{EqP.8} we get that 
\begin{equation*} 
u_{\eta_t}, \tilde{u}_{\eta_t} \,\, \text{are bounded in} \,\, L^{2}([0,T];X_0^s(\Omega)).
\end{equation*}

Now we want to obtain another a priori estimate. 

Multiplying \eqref{EqP.4} by $ u_k - u_{k-1}$, integrating over $\mathbb{R}^N$ and summing from $k=1$ to $n' \leq n$, using Young's inequality we get
\begin{equation}
\label{EqP.13}
\begin{aligned}
& \eta_t \sum_{k=1}^{n'} \int_{\Omega} \Big(\frac{u_k-u_{k-1}}{\eta_t}\Big)^2 \, dx + \sum_{k=1}^{n'} \int_{\mathbb{R}^N} \Big( (-\Delta )^{s} u_k(x) \Big) (u_k-u_{k-1})(x) \, dx \\
& \quad - \lambda \sum_{k=1}^{n'} \int_{\Omega} \frac{u_k(u_k-u_{k-1})}{|x|^{2s}} \, dx -\sum_{k=1}^{n'} \int_{\Omega} \frac{u_k-u_{k-1}}{u_k^{\gamma}} \, dx  \\
& = \eta_t \sum_{k=1}^{n'} \int_{\Omega} \frac{f_k(u_k-u_{k-1})}{\eta_t} \, dx \\
& \leq \frac{\eta_t}{2} \sum_{k=1}^{n'} \Bigg( \int_{\Omega} |f_k|^2 \, dx + \int_{\Omega} \Big(\frac{u_k-u_{k-1}}{\eta_t}\Big)^2 \, dx \Bigg),
\end{aligned}
\end{equation}
which implies
\begin{equation}
\label{EqP.14}
\begin{aligned}
& \frac{\eta_t}{2} \sum_{k=1}^{n'} \int_{\Omega} \Big(\frac{u_k-u_{k-1}}{\eta_t}\Big)^2 \, dx + \sum_{k=1}^{n'} \int_{\mathbb{R}^N} \Big( (-\Delta )^{s} u_k(x) \Big) (u_k-u_{k-1})(x) \, dx \\
& \quad - \lambda \sum_{k=1}^{n'} \int_{\Omega} \frac{u_k(u_k-u_{k-1})}{|x|^{2s}} \, dx -\sum_{k=1}^{n'} \int_{\Omega} \frac{u_k-u_{k-1}}{u_k^{\gamma}} \, dx \\
& \leq \frac{|\Omega|T}{2} \sup_{0\leq t \leq T} \|f(\cdot,t)\|_{L^{2}(\Omega)}^2.
\end{aligned}
\end{equation}
By using the convexity of the term $-\frac{1}{1-\gamma} \int_{\Omega} u^{1-\gamma} \, dx$, we get
\begin{equation}
\label{EqP.14.1}
\frac{1}{1-\gamma} \int_{\Omega} \Big( u_{k-1}^{1-\gamma} - u_k^{1-\gamma} \Big) \, dx \leq - \int_{\Omega} \frac{u_k-u_{k-1}}{u_k^{\gamma}} \, dx.
\end{equation}
Also, we have
\begin{equation}
\label{EqP.15}
\frac{C_{N,s}}{2} \Big( \|u_k\|^2_{X_0^s(\Omega)} -  \|u_{k-1}\|^2_{X_0^s(\Omega)} \Big) \leq \int_{\mathbb{R}^N} \Big( (-\Delta )^{s} u_k(x) \Big) (u_k-u_{k-1})(x) \, dx,
\end{equation}
and
\begin{equation}
\label{EqP.15-ADD-H}
\int_{\Omega} \frac{(u_k)^2-(u_{k-1})^2}{|x|^{2s}} \, dx \leq \int_{\Omega} \frac{u_k(u_k-u_{k-1})}{|x|^{2s}} \, dx.
\end{equation}
Therefore \eqref{EqP.14} together with \eqref{EqP.14.1}, \eqref{EqP.15} and \eqref{EqP.15-ADD-H} gives 
\begin{equation}
\label{EqP.16}
\begin{aligned}
& \frac{\eta_t}{2} \sum_{k=1}^{n'} \int_{\Omega} \Big(\frac{u_k-u_{k-1}}{\eta_t}\Big)^2 \, dx + \frac{C_{N,s}}{2} \Big( \|u_{n'}\|^2_{X_0^s(\Omega)} -  \|u_{0}\|^2_{X_0^s(\Omega)} \Big) \\
& \quad - \lambda \int_{\Omega} \frac{(u_{n'})^2-(u_{0})^2}{|x|^{2s}} \, dx + \frac{1}{1-\gamma} \int_{\Omega} \Big( (u_0)^{1-\gamma}-(u_{n'})^{1-\gamma} \Big) \, dx \\
& \leq \frac{|\Omega|T}{2} \sup_{0\leq t \leq T} \|f(\cdot,t)\|_{L^{2}(\Omega)}^2.
\end{aligned}
\end{equation}
Integrating over $(t_{k-1},t_k)$ on both sides of \eqref{EqP.16} and using \eqref{EqP.10.1} and Hardy inequality, we get
\begin{equation*} 
\frac{\eta_t}{2} \int_{0}^{T} \int_{\Omega} \Big| \frac{\partial \tilde{u}_{\eta_t} }{\partial t} \Big|^2 \, dx dt  < + \infty,
\end{equation*}
which implies 
\begin{equation}
\label{EqP.18}
\frac{\partial \tilde{u}_{\eta_t}}{\partial t} \,\, \text{is bounded in} \,\, L^{2}(\Omega \times (0,T)) \,\, \text{uniformly in} \,\, \eta_t.
\end{equation}
Also, using the definition of $u_{\eta_t}$ and $\tilde{u}_{\eta_t}$, we obtain
\begin{equation}
\label{EqP.19}
u_{\eta_t} \,\, \text{and} \,\, {\tilde{u}_{\eta_t}} \,\, \text{are bounded in} \,\, L^{\infty}([0,T]; X_0^s(\Omega)) \,\, \text{uniformly in} \,\, \eta_t.
\end{equation}
Moreover, there exists a constant $C > 0$ (independent of $\eta_t$) such that
\begin{equation}
\label{EqP.20}
\|u_{\eta_t} - \tilde{u}_{\eta_t}\|_{L^{\infty}([0,T];L^2(\Omega))} \leq \max_{1 \leq k \leq n} \|u_k-u_{k-1}\|_{L^2(\Omega)} \leq C (\eta_t)^{\frac{1}{2}}.
\end{equation}
Now, \eqref{EqP.11} and \eqref{EqP.19}, implies 
\begin{equation*} 
u_{\eta_t} \,\, \text{and} \,\, \tilde{u}_{\eta_t} \,\, \text{are bounded in} \,\, L^{\infty}([0,T]; X_0^s(\Omega) \cap L^{\infty}(\Omega, |x|^{\beta} \, dx)) \,\, \text{uniformly in} \,\, \eta_t.
\end{equation*}
Therefore, up to a subsequence, as $\eta_t \to 0^+$ (i.e. $n \to \infty$) 
\begin{equation}
\label{EqP.22}
\begin{aligned}
& \tilde{u}_{\eta_t} \to u, \,\, \text{and} \,\, u_{\eta_t} \to v, \quad \text{weak-starly in} \,\,  L^{\infty}([0,T]; X_0^s(\Omega) \cap L^{\infty}(\Omega, |x|^{\beta} \, dx)), \\
& \frac{\partial \tilde{u}_{\eta_t}}{\partial t} \rightharpoonup \frac{\partial u}{\partial t}, \qquad \text{weakly in} \,\, L^2(\Omega \times (0,T)),
\end{aligned}
\end{equation}
where $u,v \in L^{\infty}([0,T]; X_0^s(\Omega) \cap L^{\infty}(\Omega, |x|^{\beta} \, dx))$, and $\frac{\partial u}{\partial t} \in L^2(\Omega \times (0,T)).$
From \eqref{EqP.20}, we deduce that $u \equiv v$. Also, from \eqref{EqP.10}, we get that $\underline{u} \leq u \leq \overline{u}$. Thus $u \in \mathcal{A}(\Omega_T) \cap \mathcal{U}^{\mathrm{Sing}}_{\gamma}$.

Now we want to show that $u$ is the candidate to the weak solution to \eqref{EqP.1}. By the definition of $\tilde{u}_{\eta_t}$, we see that for a.e. $x \in \Omega$, $\tilde{u}_{\eta_t}(\cdot,x) \in C([0,T])$. By \eqref{EqP.18}, we get that $\frac{\partial \tilde{u}_{\eta_t} }{\partial t}$ is bounded in $L^2(\Omega \times (0,T))$ uniformly in $\eta_t$. Also, $\{u_{\eta_t}\}$ is a bounded family in $X_0^s(\Omega)$. Now, let define 
$$ V:=\Bigg\{ u \in C([0,T];X_0^s(\Omega)) \,\, : \,\, \frac{\partial u}{\partial t} \in L^2(\Omega \times (0,T)) \Bigg\}, $$
which embeds compactly in $C([0,T];L^2(\Omega))$, by invoking the Aubin-Lions-Simon Lemma. Therefore, we obtain that $\{u_{\eta_t}\}$ is compactly embedded in the space $C([0,T];L^2(\Omega))$. Now, using $\underline{u} \leq \tilde{u}_{\eta_t} \leq \overline{u}$, we deduce that $\{u_{\eta_t}\}$ is compactly embedded in $C([0,T];L^p(\Omega))$, $1<p<\infty$. Thus, up to a subsequence, as $\eta_t \to 0^+$ 
\begin{equation}
\label{EqP.23}
\tilde{u}_{\eta_t} \to u, \qquad \mathrm{in} \,\,  C([0,T]; L^2(\Omega)).
\end{equation}
Therefore, from \eqref{EqP.23} and \eqref{EqP.20} we obtain that as $\eta_t \to 0^+$
\begin{equation}
\label{EqP.24}
u_{\eta_t} \to u, \qquad \mathrm{in} \,\,  L^{\infty}([0,T]; L^2(\Omega)).
\end{equation}
Plugging in the test function $\phi=u_{\eta_t} - u$, in \eqref{EqP.6}, we obtain
\begin{equation*} 
\int_{0}^{T} \int_{\Omega} \Bigg(\frac{\partial \tilde{u}_{\eta_t}}{\partial t} + \Big((-\Delta )^s u_{\eta_t} - \lambda \dfrac{u_{\eta_t}}{|x|^{2s}} - \dfrac{1}{u_{\eta_t}^{\gamma}} \Big)\Bigg) (u_{\eta_t} - u) \, dx dt  = \int_{0}^{T} \int_{\Omega} f_{\eta_t} (u_{\eta_t} - u) \, dx dt.
\end{equation*}
Also, since \eqref{EqP.24} implies that $\int_{0}^{T} \int_{\Omega} \frac{\partial u}{\partial t} (\tilde{u}_{\eta_t} - u) \, dx dt  \to 0$, as $\eta_t \to 0^+$, we get
\begin{equation}
\label{EqP.25}
\begin{aligned}
\int_{0}^{T} \int_{\Omega}  & \Bigg(\frac{\partial \tilde{u}_{\eta_t}}{\partial t} - \frac{\partial u}{\partial t} \Bigg) (\tilde{u}_{\eta_t} - u) \, dx dt + \int_{0}^{T} \Big\langle (-\Delta )^s u_{\eta_t}, u_{\eta_t} - u \Big\rangle \, dt \\
& - \lambda \int_{0}^{T} \int_{\Omega} \dfrac{u_{\eta_t}(u_{\eta_t} - u)}{|x|^{2s}} \, dx dt - \int_{0}^{T} \int_{\Omega} \frac{u_{\eta_t} - u}{u_{\eta_t}^{\gamma}}\, dx dt \\
& = \int_{0}^{T} \int_{\Omega} f_{\eta_t} (u_{\eta_t} - u) \, dx dt + \mathrm{o}_{\eta_t}(1).
\end{aligned}
\end{equation}
Here $\langle \cdot, \cdot \rangle$ denotes the duality pairing between $X^{-s}(\Omega)$ and $X_0^s(\Omega)$. By \eqref{EqP.10}, we know that $u_{\eta_t}^{\gamma} \leq \underline{u}^{\gamma}$. Also, since $\underline{u} \leq u \leq \overline{u}$, by applying the Dominated Convergence Theorem, from \eqref{EqP.24}, we get 
\begin{equation*} 
\int_{0}^{T}\int_{\Omega} \frac{u_{\eta_t}-u}{u_{\eta_t}^{\gamma}} \, dx dt  \leq \int_{0}^{T}\int_{\Omega} \frac{u_{\eta_t}-u}{\underline{u}^{\gamma}} \, dx dt = \mathrm{o}_{\eta_t}(1).
\end{equation*}
Similarly, by using the Dominated Convergence Theorem, from \eqref{EqP.3} and \eqref{EqP.24}, we obtain
\begin{equation*} 
\int_{0}^{T} \int_{\Omega} f_{\eta_t} (u_{\eta_t} - u) \, dx dt = \mathrm{o}_{\eta_t}(1).
\end{equation*}
Now by noting that $\tilde{u}_{\eta_t}(x,0)=u(x,0)=u_0$ in a.e. $\Omega$, and applying the integration by parts formula, we have
\begin{equation*} 
 2 \int_{0}^{T} \int_{\Omega} \Bigg(\frac{\partial \tilde{u}_{\eta_t}}{\partial t} - \frac{\partial u}{\partial t} \Bigg) (\tilde{u}_{\eta_t} - u) \, dx dt = \int_{\Omega}  (\tilde{u}_{\eta_t} - u)^2(T) \, dt.
 \end{equation*}
Therefore, by using \eqref{EqP.25} and the facts that $\int_{0}^{T} \Big\langle (-\Delta )^s u, u_{\eta_t} - u \Big\rangle \, dt=\mathrm{o}_{\eta_t}(1)$, and $\int_{0}^{T} \int_{\Omega} \dfrac{u(u_{\eta_t} - u)}{|x|^{2s}} \, dx dt=\mathrm{o}_{\eta_t}(1)$, which they follows from \eqref{EqP.24}, we obtain
\begin{equation*} 
\begin{aligned}
\frac{1}{2} \int_{\Omega}  (\tilde{u}_{\eta_t} - u)^2(T) \, dt & + \int_{0}^{T} \Big\langle (-\Delta )^s u_{\eta_t}-(-\Delta )^s u, u_{\eta_t} - u \Big\rangle \, dt \\
& - \lambda \int_{0}^{T} \int_{\Omega} \dfrac{(u_{\eta_t} - u)^2}{|x|^{2s}} \, dx dt  = \mathrm{o}_{\eta_t}(1). 
\end{aligned}
\end{equation*}
Now, \eqref{EqP.24} together with the Hardy inequality gives
\begin{equation*} 
\int_{0}^{T} \| (\tilde{u}_{\eta_t} - u)(t, \cdot) \|_{X_0^s(\Omega)}^2 \, dt = \mathrm{o}_{\eta_t}(1).
\end{equation*}
The above equations implies that as $\eta_t \to 0^+$
\begin{equation}
\label{EqP.26}
(-\Delta)^s u_{\eta_t} \to (-\Delta)^s u, \qquad \mathrm{in} \,\, L^2([0,T]; X^{-s}(\Omega)).
\end{equation}
Using \eqref{EqP.10} and the fractional Hardy-Sobolev inequality, we obtain that the following inequalities holds for any $\phi \in X_0^s(\Omega)$.
\begin{equation*} 
\begin{aligned}
& \int_{\Omega} \Big| \frac{\phi}{u_{\eta_t}^{\gamma}} \Big| \, dx \\
& \leq 
\begin{cases}
\displaystyle\int_{\Omega} \frac{|\phi|}{|\overline{u}^{\gamma}|} \, dx \leq C \Bigg( \displaystyle\int_{\Omega} \frac{\phi^2}{\delta^{2s\gamma}} \, dx \Bigg)^{\frac{1}{2}} < + \infty, \,\, & 0<\gamma \leq 1, \\
\displaystyle\int_{\Omega} \frac{|\phi|}{|\underline{u}^{\gamma}|} \, dx \leq \Bigg( \displaystyle\int_{\Omega} \frac{1}{\delta^{2s\frac{\gamma-1}{\gamma+1}}} \, dx \Bigg)^{\frac{1}{2}} \Bigg( \displaystyle\int_{\Omega} \frac{\phi^2}{\delta^{2s}} \, dx \Bigg)^{\frac{1}{2}} < + \infty, \,\, & \gamma>1, \, 2s\frac{\gamma-1}{\gamma+1}<1.
\end{cases}
\end{aligned}
\end{equation*}
Therefore, the Dominated Convergence Theorem implies
\begin{equation}
\label{EqP.27}
\frac{1}{u_{\eta_t}^{\gamma}} \to \frac{1}{u^{\gamma}}, \qquad \mathrm{in} \,\, L^{\infty}([0,T]; X^{-s}(\Omega)) \,\, \text{as} \,\, \eta_t \to 0^+.
\end{equation}

Now we want to show that $u$ satisfies \eqref{EqP.1} in the weak sense. We already know that
\begin{equation*} 
\begin{aligned}
 \int_{0}^{T} \int_{\Omega} \frac{\partial \tilde{u}_{\eta_t}}{\partial t} \phi \, dx dt & + \int_{0}^{T} \int_{\mathbb{R}^N} (-\Delta)^s u_{\eta_t} \phi \, dx dt - \lambda \int_{0}^{T} \int_{\Omega} \dfrac{u_{\eta_t} \phi}{|x|^{2s}} \, dx dt \\
 & -\int_{0}^{T} \int_{\Omega} \frac{\phi}{u_{\eta_t}^{\gamma}} \, dx dt=\int_{0}^{T} \int_{\Omega} f_{\eta_t} \phi \, dx dt,
 \end{aligned}
\end{equation*}
holds for any $\phi \in \mathcal{A}(\Omega_T)$. Now passing on the limit $\eta_t \to 0^+$, and using \eqref{EqP.3}, \eqref{EqP.22}, \eqref{EqP.26} and \eqref{EqP.27}, we obtain 
\begin{equation*} 
\begin{aligned}
\int_{0}^{T} \int_{\Omega} \frac{\partial u}{\partial t} \phi \, dx dt & + \int_{0}^{T} \int_{\mathbb{R}^N} (-\Delta)^s u \phi \, dx dt - \lambda \int_{0}^{T} \int_{\Omega} \dfrac{u \phi}{|x|^{2s}} \, dx dt \\
& -\int_{0}^{T} \int_{\Omega} \frac{\phi}{u^{\gamma}} \, dx dt=\int_{0}^{T} \int_{\Omega} f \phi \, dx dt.
\end{aligned}
\end{equation*}
This means that, $u$ is the weak solution to \eqref{EqP.1}. 

Now we show the uniqueness. Let $u(\cdot,t), v(\cdot,t) \in X_0^s(\Omega) \cap \mathcal{U}^{\mathrm{Sing}}_{\gamma}$ be two weak solutions. Then for any $t \in [0,T]$, we have
\begin{equation*} 
\begin{aligned}
\int_{\Omega} & \frac{\partial(u-v)}{\partial t} \big(u-v \big)(x,t) \, dx + \int_{\mathbb{R}^N} \Big( (-\Delta)^s(u-v)\Big) \big(u-v \big)(x,t) \, dx \\
& - \lambda \int_{\Omega} \frac{\big(u-v \big)^2(x,t)}{|x|^{2s}} \, dx  - \int_{\Omega} \Big( \frac{1}{u^{\gamma}}-\frac{1}{v^{\gamma}} \Big) \big(u-v \big)(x,t) \ dx=0.
\end{aligned}
\end{equation*}
Using Hardy inequality, this implies:
\begin{equation*} 
\begin{aligned}
\frac{\partial }{\partial t} \Bigg( \int_{\Omega} \frac{1}{2} \big(u-v \big)^2(x,t) \, dx \Bigg) & = \frac{\Lambda_{N,s}-\lambda}{\Lambda_{N,s}} \cdot \frac{C_{N,s}}{2} \Big\|\big(u-v \big)(\cdot,t) \Big\|_{X_0^s(\Omega)}^2 \\
& \quad + \int_{\Omega} \Big( \frac{1}{u^{\gamma}}-\frac{1}{v^{\gamma}} \Big) \big(u-v \big)(x,t) \ dx \leq 0.
\end{aligned}
\end{equation*}
Therefore, the function $E : [0, T] \to \mathbb{R}$, $ E(t):= \int_{\Omega} \frac{1}{2} \big(u-v \big)^2(x,t) \, dx $, is a decreasing function. On the other hand, since $u \not \equiv v$, we get $0 < E(t) \leq E(0)=0$, which implies $E(t) \equiv 0$, for all $t \in [0,T]$. This completes the proof of uniqueness.

Now, we prove that $u \in C([0,T]; X_0^s(\Omega))$. From \eqref{EqP.23} we already know that $u \in C([0,T]; L^2(\Omega))$, which implies that the map $\tilde{u}: [0,T] \to X_0^s(\Omega)$, $\big[\tilde{u}(t)\big](x):= u(x,t)$ is weakly continuous. Moreover, from \eqref{EqP.22} we know that $u \in L^{\infty}([0,T]; X_0^s(\Omega))$, which implies $ \tilde{u}(t) \in X_0^s(\Omega)$ and 
\begin{equation}
\label{EqP.provingc.0000}
\| \tilde{u}(t) \|_{X_0^s(\Omega)} \leq \liminf_{t \to t_0} \| \tilde{u}(t) \|_{X_0^s(\Omega)},
\end{equation}
for all $t_0 \in [0,T]$. 

Now, we continue as follows. Multiplying \eqref{EqP.4} by $ u_k - u_{k-1}$, integrating over $\mathbb{R}^N$ and summing from $k=n''$ to $n'$ ($n'$ has been considered in \eqref{EqP.13}) and using \eqref{EqP.15-ADD-H}, \eqref{EqP.15} and \eqref{EqP.14.1} we get
\begin{equation}
\label{EqP.provingc1}
\begin{aligned}
& \frac{\eta_t}{2} \sum_{k=n''}^{n'} \int_{\Omega} \Big(\frac{u_k-u_{k-1}}{\eta_t}\Big)^2 \, dx + \frac{C_{N,s}}{2} \Big( \|u_{n'}\|^2_{X_0^s(\Omega)} -  \|u_{n''-1}\|^2_{X_0^s(\Omega)} \Big) \\
& \qquad \qquad  - \lambda \int_{\Omega} \frac{(u_{n'})^2-(u_{n''})^2}{|x|^{2s}} \, dx + \frac{1}{1-\gamma} \int_{\Omega} \Big( u_{n''-1}^{1-\gamma} - u_{n'}^{1-\gamma} \Big) \, dx \\
& \leq \sum_{k=n''}^{n'} \int_{\Omega} f_{\eta_t} (u_k-u_{k-1}) \, dx.
\end{aligned}
\end{equation}
For any $t_1 \in [t_0, T]$, we choose $n''$ and $n'$ in such a way that $n'' \eta_t \to t_1$ and $n' \eta_t \to t_0$ as $\eta_t \to 0^+$.
Using \eqref{EqP.3}, \eqref{EqP.20}, \eqref{EqP.24} and \eqref{EqP.27}, together with \eqref{EqP.provingc1} we get
\begin{equation}
\label{EqP.provingc2}
\begin{aligned}
& \int_{t_0}^{t_1} \int_{\Omega}  \Big( \frac{\partial u}{\partial t}\Big)^2 \, dx dt + \frac{C_{N,s}}{2} \|u(x,t_1)\|_{X_0^s(\Omega)}^2 -\lambda \int_{\Omega} \dfrac{u^2(x,t_1)}{|x|^{2s}} \, dx \\
& \qquad \qquad - \frac{1}{1-\gamma} \int_{\Omega} u^{1-\gamma}(t_1) \, dx \\
& \leq \int_{t_0}^{t_1} \int_{\Omega} f \frac{\partial u}{\partial t} \, dx dt + \frac{C_{N,s}}{2} \|u(x,t_0)\|_{X_0^s(\Omega)}^2 -\lambda \int_{\Omega} \dfrac{u^2(x,t_0)}{|x|^{2s}} \, dx \\
& \qquad \qquad - \frac{1}{1-\gamma} \int_{\Omega} u^{1-\gamma}(t_0) \, dx.
\end{aligned}
\end{equation}
Noting that $u \in L^{\infty}([0, T]; L^p(\Omega))$, for $1 <p< \infty $, we have
\begin{equation}
\label{EqP.provingc.00001}
\limsup_{t_1 \to t_0^+} \| u(\cdot ,t_1) \|_{X_0^s(\Omega)} \leq \|u(\cdot, t_0)\|_{X_0^s(\Omega)}.
\end{equation}
Therefore, \eqref{EqP.provingc.00001} together with \eqref{EqP.provingc.0000} gives $\lim_{t \to t_0^+} \| u(\cdot ,t) \|_{X_0^s(\Omega)} = \|u(\cdot, t_0)\|_{X_0^s(\Omega)}$, which implies that $u$ is right continuous on $[0, T]$.

Now it is enough to prove the left continuity. Let assume $t_1 > t_0$, and $0 < r \leq t_1 - t_0$. Define
$$ \big[\phi_r(u)\big](x,t) := \frac{u(x,t+r)-u(x,t)}{r}. $$
Using $\phi_r(u)$ as the test function in \eqref{EqP.1}, integrating over $(t_0, t_1) \times \mathbb{R}^N$ and using \eqref{EqP.14.1}, \eqref{EqP.15} and \eqref{EqP.15-ADD-H} we get
\begin{equation*} 
\begin{aligned}
&\int_{t_0}^{t_1} \int_{\Omega}  \frac{\partial u}{\partial t}  \phi_r(u) \, dx dt + \frac{C_{N,s}}{2r} \int_{t_0}^{t_1} \int_{\mathbb{R}^N} \Big( |(-\Delta)^{\frac{s}{2}} u(x,t+r)|^2 - |(-\Delta)^{\frac{s}{2}} u(x,t)|^2 \Big) \, dx dt  \\
& - \frac{\lambda}{r} \int_{t_0}^{t_1} \int_{\Omega} \dfrac{u^2(x,t+r)-u^2(x,t)}{|x|^{2s}} \, dx dt \\
& - \frac{1}{r(1-\gamma)}\int_{t_0}^{t_1}\int_{\Omega} \Big( u^{1-\gamma}(x,t +r) - u^{1-\gamma}(x,t) \Big)  \, dx dt \\
& \geq  \int_{t_0}^{t_1} \int_{\Omega} f \phi_r(u) \, dx dt.
\end{aligned}
\end{equation*}
Then an easy calculations gives
\begin{equation}
\label{EqP.provingc3}
\begin{aligned}
& \int_{t_0}^{t_1} \int_{\Omega} \frac{\partial u}{\partial t} \phi_r(u) \, dx dt + \frac{C_{N,s}}{2r} \Bigg( \int_{t_1}^{t_1+r} \int_{\mathbb{R}^N} |(-\Delta)^{\frac{s}{2}} u(x,t)|^2 \, dx dt \\
& \qquad \qquad \qquad \qquad \qquad \qquad \qquad - \int_{t_0}^{t_0+r} \int_{\mathbb{R}^N} |(-\Delta)^{\frac{s}{2}} u(x,t)|^2 \, dx dt \Bigg)  \\
 & - \frac{\lambda}{r} \Bigg( \int_{t_1}^{t_1+r} \int_{\Omega} \dfrac{u^2(x,t)}{|x|^{2s}} \, dx dt - \int_{t_0}^{t_0+r} \int_{\Omega} \dfrac{u^2(x,t)}{|x|^{2s}} \, dx dt \Bigg)   \\
& - \frac{1}{r(1-\gamma)} \Bigg( \int_{t_1}^{t_1+r} \int_{\Omega} u^{1-\gamma}(x,t) \, dx dt - \int_{t_0}^{t_0+r} \int_{\Omega} u^{1-\gamma}(x,t) \, dx dt \Bigg) \\
& \geq  \int_{t_0}^{t_1} \int_{\Omega} f \phi_r(u) \, dx dt.
\end{aligned}
\end{equation}
Since $u(t) \in X_0^s(\Omega)$ is right continuous on $[0,T]$, by using the Dominated Convergence Theorem, as $r \to 0^+$, we get:
\begin{equation*} 
\begin{aligned}
& \frac{1}{r} \int_{t_1}^{t_1+r} \int_{\mathbb{R}^N} |(-\Delta)^{\frac{s}{2}} u(x,t)|^2 \, dx dt  \to  \int_{\mathbb{R}^N} |(-\Delta)^{\frac{s}{2}} u(x,t_1)|^2 \, dx. \\
& \frac{1}{r} \int_{t_0}^{t_0+r} \int_{\mathbb{R}^N} |(-\Delta)^{\frac{s}{2}} u(x,t)|^2 \, dx dt  \to  \int_{\mathbb{R}^N} |(-\Delta)^{\frac{s}{2}} u(x,t_0)|^2 \, dx. \\
& \frac{1}{r} \int_{t_1}^{t_1+r} \int_{\Omega} \dfrac{u^2(x,t)}{|x|^{2s}} \, dx dt  \to \int_{\Omega} \dfrac{u^2(x,t_1)}{|x|^{2s}} \, dx. \\
& \frac{1}{r} \int_{t_0}^{t_0+r} \int_{\Omega} \dfrac{u^2(x,t)}{|x|^{2s}} \, dx dt  \to \int_{\Omega} \dfrac{u^2(x,t_0)}{|x|^{2s}} \, dx. \\
& \frac{1}{r} \int_{t_1}^{t_1+r} \int_{\Omega} u^{1-\gamma}(x,t) \, dx dt  \to \int_{\Omega} u^{1-\gamma}(x,t_1) \, dx.\\
& \frac{1}{r} \int_{t_0}^{t_0+r} \int_{\Omega} u^{1-\gamma}(x,t) \, dx dt \to \int_{\Omega} u^{1-\gamma}(x,t_0) \, dx.
\end{aligned}
\end{equation*}
Putting the results together in \eqref{EqP.provingc3}, as $r \to 0^+$, we obtain
\begin{equation}
\label{EqP.provingc4}
\begin{aligned}
& \int_{t_0}^{t_1} \int_{\Omega} \Big( \frac{\partial u}{\partial t}\Big)^2 \, dx dt + \frac{C_{N,s}}{2} \|u(x,t_1)\|_{X_0^s(\Omega)}^2 -\lambda \int_{\Omega} \dfrac{u^2(x,t_1)}{|x|^{2s}} \, dx \\
& \qquad \qquad - \frac{1}{1-\gamma} \int_{\Omega} u^{1-\gamma}(t_1) \, dx \\
& \geq \int_{t_0}^{t_1} \int_{\Omega} f \frac{\partial u}{\partial t} \, dx dt + \frac{C_{N,s}}{2} \|u(x,t_0)\|_{X_0^s(\Omega)}^2 -\lambda \int_{\Omega} \dfrac{u^2(x,t_0)}{|x|^{2s}} \, dx \\
& \qquad \qquad - \frac{1}{1-\gamma} \int_{\Omega} u^{1-\gamma}(t_0) \, dx.
\end{aligned}
\end{equation}
Therefore, \eqref{EqP.provingc4} and \eqref{EqP.provingc2} gives the equality. Since the maps $t \mapsto \int_{\Omega} u^{1-\gamma}(x,t) \, dt$, and $t \mapsto \int_{\Omega} \dfrac{u^2(x,t)}{|x|^{2s}} \, dx$ are continuous, therefore $u \in C([0, T]; X_0^s(\Omega))$. Moreover, \eqref{Eq.P1ESTIMATEP} obtains by taking $t_1 = t$ and $t_0 = 0$. 

Finally we want to show that, the solution obtained above can be proved to belong in $C([0, T]; W(\Omega))$ if the initial function $u_0 \in \overline{\mathcal{D}(L)}^{L^{\infty}(\Omega, |x|^{\beta} \, dx)}$. We will use the $m$-accretive operator theory. Let $u_0 \in \overline{\mathcal{D}(L)}^{L^{\infty}(\Omega, |x|^{\beta} \, dx)}$, $\theta>0$, $f_1,f_2 \in L^{\infty}(\Omega, |x|^{\beta} \, dx)$, and $0 <\lambda< \lambda_*$. Also, let $u,v \in X_0^s(\Omega) \cap \mathcal{U}^{\mathrm{Sing}}_{\gamma} \cap W(\Omega)$ be the unique solutions to
\begin{equation*} 
\begin{aligned}
& u + \theta L(u) = f_1, \qquad \text{in} \,\, \Omega, \\
& v + \theta L(v) = f_2, \qquad \text{in} \,\, \Omega.
\end{aligned}
\end{equation*}
Notice that the existence and uniqueness is guaranteed by Theorem \ref{Thm3}. Subtracting the weak formulations of these two equations and using $w:=\Big(|x|^{\beta}(u-v)-\|f_1-f_2\|_{L^{\infty}(\Omega, |x|^{\beta} \, dx)} \Big)^+$ as a test function, we obtain
\begin{equation*} 
\int_{\Omega} w^2 |x|^{\beta}\, dx + \theta \int_{\Omega} \big( L(u)-L(v) \big)w \, dx \leq 0. 
\end{equation*}
Since we can easily check that $\int_{\Omega} \big( L(u)-L(v) \big)w \, dx \geq 0$, thus $ w \equiv 0$ a.e. in $\Omega$, or equivalently
$ |x|^{\beta} (u-v) \leq \|f_1-f_2 \|_{L^{\infty}(\Omega, |x|^{\beta} \, dx)}$. Reversing the roles of $u$ and $v$ gives
\begin{equation*} 
\|u-v\|_{L^{\infty}(\Omega, |x|^{\beta} \, dx)} \leq \|f_1-f_2 \|_{L^{\infty}(\Omega, |x|^{\beta} \, dx)}.
\end{equation*}
This proves that $L$ is $m$-accretive in $W(\Omega)$. Now the rest of the proof obtains by invoking \cite[Theorem 4.2]{MR2582280}, as explained in \cite[Proposition 0.1]{MR2891356}.
\end{proof}

\begin{proof}[Proof of Theorem \ref{Thm5}]
We follow the proof of \cite[Theorem 2.12]{MR3842325}.

Let $\underline{u}, \overline{u} \in \overline{\mathcal{D}(L)}^{L^{\infty}(\Omega, |x|^{\beta} \, dx)}$ be the sub and supersolution respectively to \eqref{Eq1} with $\mu \equiv 1$ such that $\underline{u} \leq u_0 \leq \overline{u}$, which is possible because of $u_0 \in \overline{\mathcal{D}(L)}^{L^{\infty}(\Omega, |x|^{\beta} \, dx)}$.  Let $u$ denotes the weak solution of \eqref{EqP.1} and $v_1$ and $v_2$ be the unique solutions to \eqref{EqP.1} with the initial conditions $\underline{u}$ and $\overline{u}$, respectively. Since $\lambda \in (0,\lambda_*)$, and $\underline{u}, \overline{u} \in \overline{\mathcal{D}(L)}^{L^{\infty}(\Omega, |x|^{\beta} \, dx)}$, thus Theorem \ref{Thm4} gives $v_1, v_2 \in C([0,T];W(\Omega))$. Taking $\underline{u}_0=\underline{u}$ (respectively $\overline{u}_0=\overline{u}$), we consider the sequence $\{\underline{u}_k\}$ (respectively $\{\overline{u}_k\}$) which is non-decreasing (respectively non-increasing) as solutions to the iteration given by \eqref{EqP.4}.
Moreover, we consider the sequence $\{u_k\}$ as the one that is obtained in the iteration \eqref{EqP.4}, and starts with the initial condition $u_0$. Then by the choice of $\eta_t$ we may have
\begin{equation*} 
\underline{u}_k \leq u_k \leq \overline{u}_k,
\end{equation*}
which implies
\begin{equation}
\label{EqP.29.0}
v_1(t) \leq u(t) \leq v_2(t).
\end{equation} 
Now consider the maps $t \mapsto v_1(x,t)$ and $ t \mapsto v_2(x,t)$, which are non-decreasing and non-increasing, respectively (by similar reasoning as the one in \cite[Lemma 10.6]{MR1301779}, or the proof of \cite[Theorem 2.10]{MR3842325}). Also, let $v_1(t) \to \tilde{v}_1$ and $v_2(t) \to \tilde{v}_2$ as $t \to \infty$. Moreover, if $S(t)$ denotes the semigroup on $W(\Omega)$ generated by the given evolution equation $u_t+L(u)=f(x)$, then clearly we have
\begin{equation*} 
\tilde{v}_1=\lim_{t' \to \infty} S(t'+t)(\underline{u}) = S(t) \lim_{t' \to \infty} S(t')(\underline{u}) = S(t) \lim_{t' \to \infty} v_1(t') = S(t) \tilde{v}_1.
\end{equation*}
Similarly, we get
\begin{equation*} 
\tilde{v}_2=S(t) \tilde{v}_2.
\end{equation*}
Thus $\tilde{v}_1$ and $\tilde{v}_2$ are the stationary solutions to \eqref{EqP.1} i.e. solves \eqref{Eq1} with $\mu \equiv 1$. On the other hand, by the uniqueness of solutions to the stationary problem, $\tilde{v}_1=\tilde{v}_2=\hat{u}$. Now, applying the Dini's Theorem (see \cite[Theorem 7.13]{MR0385023}) gives
\begin{equation*} 
\begin{cases}
v_1(t) \to \hat{u} \\
v_2(t) \to \hat{u}
\end{cases}
\quad \text{in} \,\, L^{\infty}(\Omega, |x|^{\beta} \, dx) \,\, \text{as} \,\, t \to \infty.
\end{equation*}
Finally, using \eqref{EqP.29.0}, we conclude that $u(t) \to \hat{u}$ in $L^{\infty}(\Omega, |x|^{\beta} \, dx)$, as $t \to \infty$. 
\end{proof}

\bibliographystyle{plainnat}
\bibliography{mybibfile}

\begin{thebibliography}{10}
\expandafter\ifx\csname url\endcsname\relax
  \def\url#1{\texttt{#1}}\fi
\expandafter\ifx\csname urlprefix\endcsname\relax\def\urlprefix{URL }\fi
\expandafter\ifx\csname href\endcsname\relax
  \def\href#1#2{#2} \def\path#1{#1}\fi

\bibitem{MR2545987}
A.~Boumediene, I.~Peral, A.~Primo, Influence of the {H}ardy potential in a
  semilinear heat equation, Proc. Roy. Soc. Edinburgh Sect. A 139~(5) (2009)
  897--926.

\bibitem{MR3258136}
M.~P\'{e}rez-Llanos, A.~Primo, Semilinear biharmonic problems with a singular
  term, J. Differential Equations 257~(9) (2014) 3200--3225.

\bibitem{MR3967804}
N.~Abatangelo, E.~Valdinoci, Getting acquainted with the fractional
  {L}aplacian, in: Contemporary research in elliptic {PDE}s and related topics,
  Vol.~33 of Springer INdAM Ser., Springer, Cham, 2019, pp. 1--105.

\bibitem{MR2944369}
E.~Di~Nezza, G.~Palatucci, E.~Valdinoci, Hitchhiker's guide to the fractional
  {S}obolev spaces, Bull. Sci. Math. 136~(5) (2012) 521--573.

\bibitem{MR2270163}
L.~Silvestre, Regularity of the obstacle problem for a fractional power of the
  {L}aplace operator, Comm. Pure Appl. Math. 60~(1) (2007) 67--112.

\bibitem{MR3468562}
M.~Ghergu, S.~D. Taliaferro, Isolated singularities in partial differential
  inequalities, Vol. 161 of Encyclopedia of Mathematics and its Applications,
  Cambridge University Press, Cambridge, 2016.

\bibitem{MR2488149}
M.~Ghergu, V.~D. R\u{a}dulescu, Singular elliptic problems: bifurcation and
  asymptotic analysis, Vol.~37 of Oxford Lecture Series in Mathematics and its
  Applications, The Clarendon Press, Oxford University Press, Oxford, 2008.

\bibitem{MR2048513}
F.~Gazzola, H.-C. Grunau, E.~Mitidieri, Hardy inequalities with optimal
  constants and remainder terms, Trans. Amer. Math. Soc. 356~(6) (2004)
  2149--2168.

\bibitem{MR1769903}
E.~Mitidieri, A simple approach to {H}ardy inequalities, Mat. Zametki 67~(4)
  (2000) 563--572.

\bibitem{HardyLerayBook}
I.~P. Alonso, F.~S. de~Diego, Elliptic and Parabolic Equations Involving the
  Hardy-Leray Potential, Walter de Gruyter GmbH \& Co KG, Berlin, 2021.

\bibitem{Biccari}
U.~Biccari, On the controllability of partial differential equations involving
  non-local terms and singular potentials, Dissertation, Universidad del País
  Vasco-Euskal Herriko Unibertsitatea.

\bibitem{MR742415}
P.~Baras, J.~A. Goldstein, The heat equation with a singular potential, Trans.
  Amer. Math. Soc. 284~(1) (1984) 121--139.

\bibitem{MR799330}
P.~Baras, J.~A. Goldstein, Remarks on the inverse square potential in quantum
  mechanics, in: Differential equations ({B}irmingham, {A}la., 1983), Vol.~92
  of North-Holland Math. Stud., North-Holland,, Amsterdam, 1984, pp. 31--35.

\bibitem{MR1717839}
D.~Yafaev, Sharp constants in the {H}ardy-{R}ellich inequalities, J. Funct.
  Anal. 168~(1) (1999) 121--144.

\bibitem{MR2592976}
L.~Boccardo, L.~Orsina, Semilinear elliptic equations with singular
  nonlinearities, Calc. Var. Partial Differential Equations 37~(3-4) (2010)
  363--380.

\bibitem{MR3450747}
L.~Orsina, F.~Petitta, A {L}azer-{M}c{K}enna type problem with measures,
  Differential Integral Equations 29~(1-2) (2016) 19--36.

\bibitem{MR3356049}
B.~n. Barrios, I.~De~Bonis, M.~Medina, I.~Peral, Semilinear problems for the
  fractional laplacian with a singular nonlinearity, Open Math. 13~(1) (2015)
  390--407.

\bibitem{MR3639996}
A.~Canino, L.~Montoro, B.~Sciunzi, M.~Squassina, Nonlocal problems with
  singular nonlinearity, Bull. Sci. Math. 141~(3) (2017) 223--250.

\bibitem{MR3797738}
L.~M. De~Cave, R.~Durastanti, F.~Oliva, Existence and uniqueness results for
  possibly singular nonlinear elliptic equations with measure data, NoDEA
  Nonlinear Differential Equations Appl. 25~(3) (2018) Art. 18, 35.

\bibitem{MR3943307}
J.~Giacomoni, T.~Mukherjee, K.~Sreenadh, Existence of three positive solutions
  for a nonlocal singular {D}irichlet boundary problem, Adv. Nonlinear Stud.
  19~(2) (2019) 333--352.

\bibitem{MR1037213}
A.~C. Lazer, P.~J. McKenna, On a singular nonlinear elliptic boundary-value
  problem, Proc. Amer. Math. Soc. 111~(3) (1991) 721--730.

\bibitem{MR3323892}
G.~Molica~Bisci, D.~Repov\v{s}, Existence and localization of solutions for
  nonlocal fractional equations, Asymptot. Anal. 90~(3-4) (2014) 367--378.

\bibitem{MR3489386}
F.~Oliva, F.~Petitta, On singular elliptic equations with measure sources,
  ESAIM Control Optim. Calc. Var. 22~(1) (2016) 289--308.

\bibitem{MR3712944}
F.~Oliva, F.~Petitta, Finite and infinite energy solutions of singular elliptic
  problems: existence and uniqueness, J. Differential Equations 264~(1) (2018)
  311--340.

\bibitem{MR4244929}
R.~Arora, J.~Giacomoni, G.~Warnault, Regularity results for a class of
  nonlinear fractional {L}aplacian and singular problems, NoDEA Nonlinear
  Differential Equations Appl. 28~(3) (2021) Paper No. 30, 35.
\newblock \href {http://dx.doi.org/10.1007/s00030-021-00693-9}
  {\path{doi:10.1007/s00030-021-00693-9}}.

\bibitem{MADD1}
L.~Marta, F.~Oliva, F.~Petitta, S.~S. de~Le\'on, The dirichlet problem for the
  $1$-laplacian with a general singular term and {$L^ 1$}-data, Nonlinearity
  34~(3) (2021) 1791--1816.

\bibitem{MR564014}
A.~Nachman, A.~J. Callegari, A nonlinear singular boundary value problem in the
  theory of pseudoplastic fluids, SIAM J. Appl. Math. 38~(2) (1980) 275--281.
\newblock \href {http://dx.doi.org/10.1137/0138024}
  {\path{doi:10.1137/0138024}}.

\bibitem{MR3466525}
T.~Mukherjee, K.~Sreenadh, Fractional elliptic equations with critical growth
  and singular nonlinearities, Electron. J. Differential Equations (2016) Paper
  No. 54, 23.

\bibitem{MR3479207}
A.~Boumediene, M.~Medina, I.~Peral, A.~Primo, The effect of the {H}ardy
  potential in some {C}alder\'{o}n-{Z}ygmund properties for the fractional
  {L}aplacian, J. Differential Equations 260~(11) (2016) 8160--8206.

\bibitem{MR3842325}
J.~Giacomoni, T.~Mukherjee, K.~Sreenadh, Existence and stabilization results
  for a singular parabolic equation involving the fractional {L}aplacian,
  Discrete Contin. Dyn. Syst. Ser. S 12~(2) (2019) 311--337.

\bibitem{MR2891356}
M.~Badra, K.~Bal, J.~Giacomoni, A singular parabolic equation: existence,
  stabilization, J. Differential Equations 252~(9) (2012) 5042--5075.

\bibitem{MR3341461}
B.~Bougherara, J.~Giacomoni, Existence of mild solutions for a singular
  parabolic equation and stabilization, Adv. Nonlinear Anal. 4~(2) (2015)
  123--134.

\bibitem{Giacomoni001}
J.~Giacomoni, D.~Goel, K.~Sreenadh, Singular doubly nonlocal elliptic problems
  with choquard type critical growth nonlinearities, arXiv:2002.02937.

\bibitem{Giacomoni002}
J.~Giacomoni, D.~Kumar, K.~Sreenadh, Sobolev and {H}\"older regularity results
  for some singular double phase problems, arXiv:2004.06699.

\bibitem{MR2879266}
R.~Servadei, E.~Valdinoci, Mountain pass solutions for non-local elliptic
  operators, J. Math. Anal. Appl. 389~(2) (2012) 887--898.

\bibitem{MR3341104}
G.~Molica~Bisci, D.~Repov\v{s}, On doubly nonlocal fractional elliptic
  equations, Atti Accad. Naz. Lincei Rend. Lincei Mat. Appl. 26~(2) (2015)
  161--176.

\bibitem{MR3168912}
X.~Ros-Oton, J.~Serra, The {D}irichlet problem for the fractional {L}aplacian:
  regularity up to the boundary, J. Math. Pures Appl. (9) 101~(3) (2014)
  275--302.

\bibitem{MR3002745}
R.~Servadei, E.~Valdinoci, Variational methods for non-local operators of
  elliptic type, Discrete Contin. Dyn. Syst. 33~(5) (2013) 2105--2137.

\bibitem{MR4065090}
B.~n. Barrios, M.~Medina, Strong maximum principles for fractional elliptic and
  parabolic problems with mixed boundary conditions, Proc. Roy. Soc. Edinburgh
  Sect. A 150~(1) (2020) 475--495.
\newblock \href {http://dx.doi.org/10.1017/prm.2018.77}
  {\path{doi:10.1017/prm.2018.77}}.

\bibitem{MR3231059}
B.~Barrios, M.~Medina, I.~Peral, Some remarks on the solvability of non-local
  elliptic problems with the {H}ardy potential, Commun. Contemp. Math. 16~(4)
  (2014) 1350046, 29.

\bibitem{MR2257147}
L.~Boccardo, L.~Orsina, I.~Peral, A remark on existence and optimal summability
  of solutions of elliptic problems involving {H}ardy potential, Discrete
  Contin. Dyn. Syst. 16~(3) (2006) 513--523.

\bibitem{MR3492734}
A.~Boumediene, M.~Medina, I.~Peral, A.~Primo, Optimal results for the
  fractional heat equation involving the {H}ardy potential, Nonlinear Anal. 140
  (2016) 166--207.

\bibitem{MR3890060}
N.~S. Papageorgiou, V.~D. R\u{a}dulescu, D.~D. Repov\v{s}, Nonlinear
  analysis---theory and methods, Springer Monographs in Mathematics, Springer,
  Cham, 2019.

\bibitem{MR3393266}
T.~Leonori, I.~Peral, A.~Primo, F.~Soria, Basic estimates for solutions of a
  class of nonlocal elliptic and parabolic equations, Discrete Contin. Dyn.
  Syst. 35~(12) (2015) 6031--6068.

\bibitem{MR3186917}
A.~Boumediene, I.~Peral, A.~Primo, A remark on the fractional {H}ardy
  inequality with a remainder term, C. R. Math. Acad. Sci. Paris 352~(4) (2014)
  299--303.

\bibitem{GVPF}
C.~Huyuan, P.~Felmer, L.~V\'{e}ron, Elliptic equations involving general
  subcritical source nonlinearity and measures, arXiv:1409.3067.

\bibitem{MR3156649}
M.~Marcus, L.~V\'{e}ron, Nonlinear second order elliptic equations involving
  measures, Vol.~21 of De Gruyter Series in Nonlinear Analysis and
  Applications, De Gruyter, Berlin, 2014.

\bibitem{MR0290095}
E.~M. Stein, Singular integrals and differentiability properties of functions,
  Princeton Mathematical Series, No. 30, Princeton University Press, Princeton,
  N.J., 1970.

\bibitem{MR3797614}
Adimurthi, J.~Giacomoni, S.~Santra, Positive solutions to a fractional equation
  with singular nonlinearity, J. Differential Equations 265~(4) (2018)
  1191--1226.

\bibitem{MR3814363}
L.~Brasco, E.~Cinti, On fractional {H}ardy inequalities in convex sets,
  Discrete Contin. Dyn. Syst. 38~(8) (2018) 4019--4040.

\bibitem{MR3933834}
A.~Boumediene, K.~Biroud, A.~Primo, Nonlinear fractional elliptic problem with
  singular term at the boundary, Complex Var. Elliptic Equ. 64~(6) (2019)
  909--932.

\bibitem{MR2085428}
B.~Dyda, A fractional order {H}ardy inequality, Illinois J. Math. 48~(2) (2004)
  575--588.

\bibitem{MR3021545}
S.~Filippas, L.~Moschini, A.~Tertikas, Sharp trace {H}ardy-{S}obolev-{M}az'ya
  inequalities and the fractional {L}aplacian, Arch. Ration. Mech. Anal.
  208~(1) (2013) 109--161.

\bibitem{MR1354907}
P.~B\'{e}nilan, L.~Boccardo, T.~Gallou\"{e}t, R.~Gariepy, M.~Pierre, J.~L.
  V\'{a}zquez, An {$L^1$}-theory of existence and uniqueness of solutions of
  nonlinear elliptic equations, Ann. Scuola Norm. Sup. Pisa Cl. Sci. (4) 22~(2)
  (1995) 241--273.

\bibitem{MR1409661}
L.~Boccardo, T.~Gallou\"{e}t, L.~Orsina, Existence and uniqueness of entropy
  solutions for nonlinear elliptic equations with measure data, Ann. Inst. H.
  Poincar\'{e} Anal. Non Lin\'{e}aire 13~(5) (1996) 539--551.

\bibitem{MR916688}
J.~Simon, Compact sets in the space {$L^p(0,T;B)$}, Ann. Mat. Pura Appl. (4)
  146 (1987) 65--96.

\bibitem{MR2722059}
M.~Badiale, E.~Serra, Semilinear elliptic equations for beginners,
  Universitext, Springer, London, 2011, existence results via the variational
  approach.

\bibitem{MR2582280}
V.~Barbu, Nonlinear differential equations of monotone types in {B}anach
  spaces, Springer Monographs in Mathematics, Springer, New York, 2010.

\bibitem{MR1301779}
J.~Smoller, Shock waves and reaction-diffusion equations, 2nd Edition, Vol. 258
  of Grundlehren der Mathematischen Wissenschaften [Fundamental Principles of
  Mathematical Sciences], Springer-Verlag, New York, 1994.
\newblock \href {http://dx.doi.org/10.1007/978-1-4612-0873-0}
  {\path{doi:10.1007/978-1-4612-0873-0}}.

\bibitem{MR0385023}
W.~Rudin, Principles of mathematical analysis, 3rd Edition, McGraw-Hill Book
  Co., New York-Auckland-D\"{u}sseldorf, 1976, international Series in Pure and
  Applied Mathematics.

\end{thebibliography}

\end{document}